\documentclass[11pt]{article}
\usepackage{amsmath}
\usepackage{amssymb,amsbsy,amsthm}
\usepackage{graphicx}
\usepackage{enumerate}
\usepackage{cases}
\usepackage[margin=1in]{geometry}
\usepackage{pdfsync}
\usepackage{hyperref}
\usepackage{wrapfig} 

\allowdisplaybreaks[1]

\numberwithin{equation}{section}

\DeclareMathOperator{\R}{\mathbb{R}} 
\DeclareMathOperator{\Z}{\mathbb{Z}} 

\newcommand{\p}{\mathbb{P}} 
\renewcommand{\P}{\mathbb{P}} 
\newcommand{\E}{\mathbb{E}} 
\newcommand{\eps}{\varepsilon} 
\DeclareMathOperator{\Var}{Var} 
\DeclareMathOperator{\Cov}{Cov} 
\newcommand{\TV}{\mathrm{TV}} 
\newcommand{\Tr}{\mathrm{Tr}} 

\DeclareMathOperator{\Bin}{Bin} 

\newcommand{\Beta}{\mathrm{Beta}} 
\newcommand{\Dir}{\mathrm{Dir}} 
\newcommand{\Poi}{\mathrm{Poi}} 

\newcommand{\diag}{\mathrm{diag}} 

\DeclareMathOperator*{\argmax}{arg\,max} 

\newcommand{\SBM}{\mathrm{SBM}} 

\newcommand{\PA}{\mathrm{PA}} 
\newcommand{\UA}{\mathrm{UA}} 

\def\cL{{\mathcal L}}

\def\cN{{\mathcal N}}

\def\cP{{\mathcal P}}

\newtheorem{theorem}{Theorem}[section]
\newtheorem{lemma}[theorem]{Lemma}

\newtheorem{conjecture}[theorem]{Conjecture}
\newtheorem{question}[theorem]{Question}

\newtheorem{example}[theorem]{Example}
\newtheorem{exercise}{Exercise}[section]

\begin{document}

\title{Basic models and questions in statistical network analysis}
\author{
	Mikl\'os Z.\ R\'acz
	\thanks{Microsoft Research; \texttt{miracz@microsoft.com}.} \\
	with 
	S\'ebastien Bubeck
	\thanks{Microsoft Research; \texttt{sebubeck@microsoft.com}.}
}
\date{\today}

\maketitle


\begin{abstract}
Extracting information from large graphs has become an important statistical problem since network data is now common in various fields. 
In this minicourse we will investigate the most natural statistical questions for three canonical probabilistic models of networks: 
(i) community detection in the stochastic block model, 
(ii) finding the embedding of a random geometric graph, and 
(iii) finding the original vertex in a preferential attachment tree. 
Along the way we will cover many interesting topics in probability theory such as 
P\'olya urns, large deviation theory, concentration of measure in high dimension, entropic central limit theorems, and more. 

Outline:
\begin{itemize}
 \item Lecture 1: A primer on exact recovery in the general stochastic block model. 
 \item Lecture 2: Estimating the dimension of a random geometric graph on a high-dimensional sphere. 
 \item Lecture 3: Introduction to entropic central limit theorems and a proof of the fundamental limits of dimension estimation in random geometric graphs. 
 \item Lectures 4 \& 5: Confidence sets for the root in uniform and preferential attachment trees.
\end{itemize}
\end{abstract}



\hfill

\subsection*{Acknowledgements} \label{sec:ack} 

These notes were prepared for a minicourse presented at University of Washington during June 6--10, 2016, and at the XX Brazilian School of Probability held at the S\~{a}o Carlos campus of Universidade de S\~{a}o Paulo during July 4--9, 2016. 
We thank the organizers of the Brazilian School of Probability, 
Paulo Faria da Veiga, 
Roberto Imbuzeiro Oliveira, 
Leandro Pimentel, 
and Luiz Renato Fontes, 
for inviting us to present a minicourse on this topic. 
We also thank Sham Kakade, Anna Karlin, and Marina Meila for help with organizing at University of Washington. 
Many thanks to all the participants who asked good questions and provided useful feedback, 
in particular Kira Goldner, Chris Hoffman, Jacob Richey, and Ryokichi Tanaka in Seattle, 
and Vladimir Belitsky, Santiago Duran, Simon Griffiths, and Roberto Imbuzeiro Oliveira in S\~{a}o Carlos. 

\newpage

\section{Lecture 1: A primer on exact recovery in the general stochastic block model} \label{sec:lec1} 

Community detection is a fundamental problem in many sciences, 
such as sociology (e.g., finding tight-knit groups in social networks), biology (e.g., detecting protein complexes), and beyond. 
Given its importance, there have been a plethora of algorithms developed in the past few decades to detect communities. 
But how can we test whether an algorithm performs well? 
What are the fundamental limits to \emph{any} community detection algorithm? 
Often in real data the ground truth is not known (or there is not even a well-defined ground truth), 
so judging the performance of algorithms can be difficult. 
Probabilistic generative models can be used to model real networks, and even if they do not fit the data perfectly, they can still be useful: 
they can act as benchmarks for comparing different clustering algorithms, since the ground truth is known.

Perhaps the most widely studied generative model that exhibits community structure is the stochastic block model (SBM). 
The SBM was first introduced in sociology~\cite{HLL83} and was then studied in several different scientific communities, including mathematics, computer science, physics, and statistics~\cite{BLCS,CondonKarp,GirvanNewman,BickelChen,KarrerNewmanDCSBM,rohe2011spectral}.\footnote{Disclaimer: the literature on community detection is vast and rapidly growing. It is not our intent here to survey this literature; we refer the interested reader to the papers we cite for further references.} 
It gives a distribution on graphs with $n$ vertices with a hidden partition of the nodes into $k$ communities. 
The relative sizes of the communities, and the edge densities connecting communities are parameters of the general SBM. 
The statistical inference problem is then to recover as much of the community structure as possible given a realization of the graph, but without knowing any of the community labels.

\subsection{The stochastic block model and notions of recovery} \label{sec:sbm} 

The general stochastic block model is a distribution on graphs with latent community structure, and it has three parameters: 
$n$, the number of vertices; 
a probability distribution $p = (p_1, \dots, p_k)$ that describes the relative sizes of the communities; 
and $Q \in \left[0,1\right]^{k \times k}$, a symmetric $k \times k$ matrix that describes the probabilities with which two given vertices are connected, depending on which communities they belong to. 
The number of communities, $k$, is implicit in this notation; in these notes we assume that $k$ is a fixed constant. 
A random graph from $\SBM( n, p, Q)$ is defined as follows: 
\begin{itemize}
 \item The vertex set of the graph is $V = \left\{ 1, \dots, n \right\} \equiv \left[ n \right]$. 

 \item Every vertex $v \in V$ is independently assigned a (hidden) label $\sigma_v \in \left[ k \right]$ from the probability distribution $p$ on $\left[k \right]$. That is, $\p \left( \sigma_v = i \right) = p_i$ for every $i \in \left[k \right]$. 

 \item Given the labels of the vertices, each (unordered) pair of vertices $\left( u, v \right) \in V \times V$ is connected independently with probability $Q_{\sigma_{u}, \sigma_{v}}$. 
\end{itemize}

\begin{figure}[h!]
 \centering
 \includegraphics[width=0.55\textwidth]{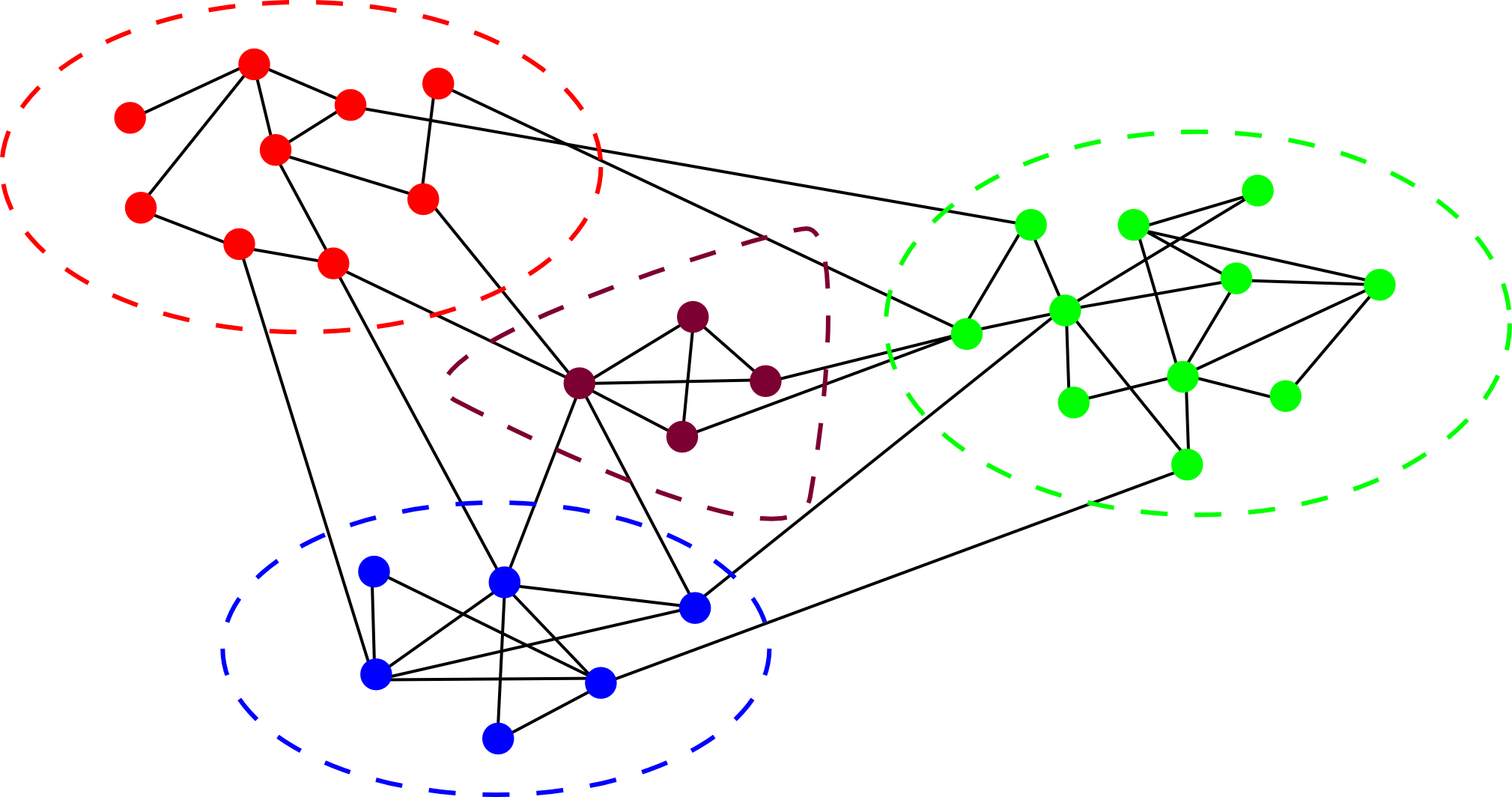}
 \caption{A schematic of the general stochastic block model.}
 \label{fig:sbm}
\end{figure}

\begin{example}[Symmetric communities]\label{ex:symm}
 A simple example to keep in mind is that of symmetric communities, with more edges within communities than between communities. 
This is modeled by the SBM with $p_i = 1/k$ for all $i \in \left[ k \right]$ and $Q_{i,j} = a$ if $i = j$ and $Q_{i,j} = b$ otherwise, with $a > b > 0$. 
\end{example}

We write $G \sim \SBM(n,p,Q)$ for a graph generated according to the SBM without the hidden vertex labels revealed. 
The goal of a statistical inference algorithm is to recover as many labels as possible using only the underlying graph as an observation. 
There are various notions of success that are worth studying. 
\begin{itemize}
 \item \textbf{Weak recovery} (also known as detection). 
An algorithm is said to \emph{weakly recover} or \emph{detect} the communities 
if it outputs a partition of the nodes which is positively correlated with the true partition, with high probability (whp)\footnote{In these notes ``with high probability'' stands for with probability tending to $1$ as the number of nodes in the graph, $n$, tends to infinity.}.

 \item \textbf{Partial recovery.} 
How much can be recovered about the communities? 
An algorithm is said to \emph{recover communities with an accuracy of $\alpha \in [0,1]$} 
if it outputs a labelling of the nodes which agrees with the true labelling on a fraction $\alpha$ of the nodes whp. 
An important special case is when only $o(n)$ vertices are allowed to be misclassified whp, known as \emph{weak consistency} or \emph{almost exact recovery}.

 \item \textbf{Exact recovery} (also known as recovery or strong consistency). 
The strongest notion of reconstruction is to recover the labels of all vertices exactly whp. 
When this is not possible, it can still be of interest to understand which communities can be exactly recovered, if not all; this is sometimes known as \emph{``partial-exact-recovery''}. 

\end{itemize}

In all the notions above, the agreement of a partition with the true partition is maximized over all relabellings of the communities, since we are not interested in the specific original labelling per se, but rather the partition (community structure) it induces.

The different notions of recovery naturally lead to studying different regimes of the parameters. 
For weak recovery to be possible, many vertices in all but one community should be non-isolated (in the symmetric case this means that there should be a giant component), requiring the edge probabilities to be $\Omega \left( 1 / n \right)$. 
For exact recovery, all vertices in all but one community should be non-isolated (in the symmetric case this means that the graph should be connected), requiring the edge probabilities to be $\Omega \left( \ln(n) / n \right)$. 
In these regimes it is natural to scale the edge probability matrices accordingly, i.e., to consider $\SBM \left( n, p, Q / n \right)$ or $\SBM \left( n, p, \ln \left( n \right) Q / n \right)$, where $Q \in \R_{+}^{k \times k}$.

There has been lots of work in the past few years understanding the fundamental limits to recovery under the various notions discussed above. 
For weak recovery there is a sharp phase transition, the threshold of which was first conjectured in~\cite{DKMZ11}. This was proven first for two symmetric communities~\cite{massoulie2014community,MNS13} and then for multiple communities~\cite{AbbeSandon15detection}. Partial recovery is less well understood, and finding the fraction of nodes that can be correctly recovered for a given set of parameters is an open problem; see~\cite{MNS14belief} for results in this direction for two symmetric communities.

In this lecture we are interested in exact recovery, for which Abbe and Sandon gave the value of the threshold for the general SBM, and showed that a quasi-linear time algorithm works all the way to the threshold~\cite{AbbeSandon15} (building on previous work that determined the threshold for two symmetric communities~\cite{ABH16,MNS15}). The remainder of this lecture is an exposition of their main results and a few of the key ideas that go into proving and understanding it.

\subsection{From exact recovery to testing multivariate Poisson distributions} \label{sec:exact} 

Recall that we are interested in the logarithmic degree regime for exact recovery, i.e., we consider $G \sim \SBM( n, p, \ln(n) Q / n )$, where $Q \in \R_{+}^{k \times k}$ is independent of $n$. We also assume that the communities have linear size, i.e., that $p$ is independent of $n$, and $p_i \in (0,1)$ for all $i$. Our goal is to recover the labels of all the vertices whp.

As a thought experiment, imagine that not only is the graph $G$ given, but also all vertex labels are revealed, except for that of a given vertex $v \in V$. 
Is it possible to determine the label of $v$? 

\begin{figure}[h!]
 \centering
 \includegraphics[width=0.55\textwidth]{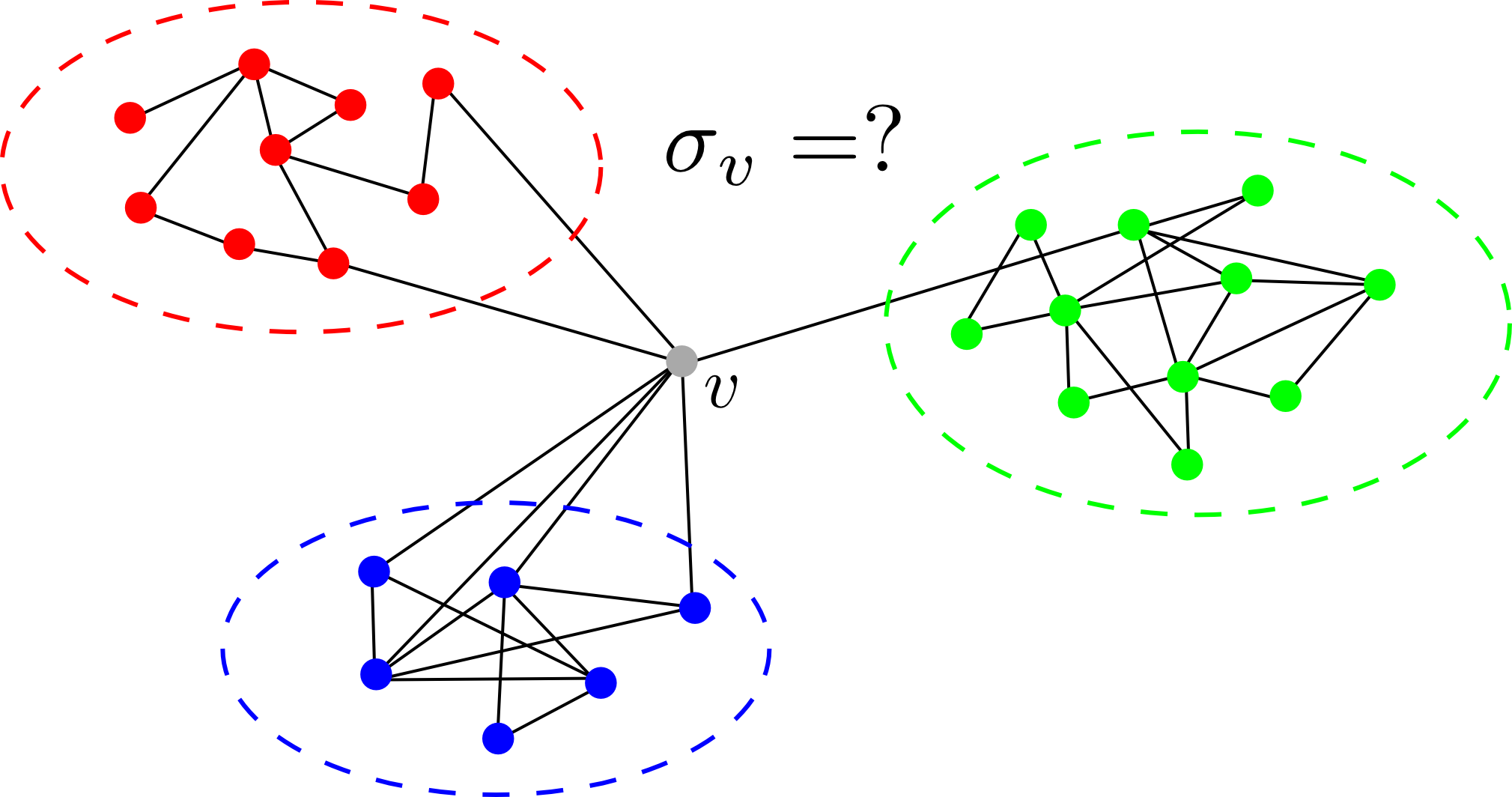}
 \caption{Suppose all community labels are known except that of vertex $v$. Can the label of $v$ be determined based on its neighbors' labels?}
 \label{fig:sbm_thought_experiment}
\end{figure}

Understanding this question is key for understanding exact recovery, 
since if the error probability of this is too high, then exact recovery will not be possible. 
On the other hand, it turns out that in this regime it is possible to recover all but $o(n)$ labels using an initial partial recovery algorithm. 
The setup of the thought experiment then becomes relevant, and if we can determine the label of $v$ given the labels of all the other nodes with low error probability, then we can correct all errors made in the initial partial recovery algorithm, leading to exact recovery. 
We will come back to the connection between the thought experiment and exact recovery; for now we focus on understanding this thought experiment.

Given the labels of all vertices except $v$, the information we have about $v$ is the number of nodes in each community it is connected to. 
In other words, we know the \emph{degree profile} $d(v)$ of $v$, where, for a given labelling of the graph's vertices, the $i$-th component $d_{i} (v)$ is the number of edges between $v$ and the vertices in community $i$. 

The distribution of the degree profile $d(v)$ depends on the community that $v$ belongs to. 
Recall that the community sizes are given by a multinomial distribution with parameters $n$ and $p$, 
and hence the relative size of community $i \in [k]$ concentrates on $p_i$. 
Thus if $\sigma_v = j$, the degree profile $d(v) = (d_1(v), \dots, d_k(v))$ can be approximated by independent binomials, with 
$d_i (v)$ approximately distributed as 
$\Bin \left( np_i, \ln(n) Q_{i,j} / n \right)$, 
where $\Bin(m,q)$ denotes the binomial distribution with $m$ trials and success probability $q$. 
In this regime, the binomial distribution is well-approximated by a Poisson distribution of the same mean. 
In particular, Le Cam's inequality gives that 
\[
 \TV \left( \Bin \left( na, \ln(n) b / n \right), \Poi \left( ab \ln (n) \right) \right) \leq \frac{2 a b^2 \left( \ln(n) \right)^2}{n},
\]
where $\Poi \left( \lambda \right)$ denotes the Poisson distribution with mean $\lambda$, 
and $\TV$ denotes the total variation distance\footnote{Recall that the total variation distance between two random variables $X$ and $Y$ taking values in a finite space $\mathcal{X}$ with laws $\mu$ and $\nu$ is defined as 
$\TV \left( \mu, \nu \right) \equiv \TV \left(X,Y \right) = \frac{1}{2} \sum_{x\in \mathcal{X}} \left| \mu \left( x \right) - \nu \left( x \right) \right| = \sup_{A} \left| \mu \left( A \right) - \nu \left( A \right) \right|$.}. 
Using the additivity of the Poisson distribution and the triangle inequality, we get that 
\[
 \TV \left( \cL \left( d(v) \right), \Poi \left( \ln \left( n \right) \sum_{i \in [k]} p_i Q_{i,j} e_i \right) \right) = O \left( \frac{\left( \ln \left( n \right) \right)^2}{n} \right),
\]
where $\cL \left( d(v) \right)$ denotes the law of $d(v)$ conditionally on $\sigma_{v} = j$ and $e_i$ is the $i$-th unit vector.

Thus the degree profile of a vertex in community $j$ is approximately Poisson distributed with mean $\ln \left( n \right) \sum_{i \in [k]} p_i Q_{i,j} e_i$. 
Defining $P = \diag(p)$, this can be abbreviated as $\ln \left( n \right) \left( PQ \right)_{j}$, where $\left( PQ \right)_{j}$ denotes the $j$-th column of the matrix $PQ$. 
We call the quantity $\left( PQ \right)_{j}$ the \emph{community profile} of community $j$; this is the quantity that determines the distribution of the degree profile of vertices from a given community.

Our thought experiment has thus been reduced to a Bayesian hypothesis testing problem between $k$ multivariate Poisson distributions. 
The prior on the label of $v$ is given by $p$, and we get to observe the degree profile $d(v)$, which comes from one of $k$ multivariate Poisson distributions, which have mean $\ln(n)$ times the community profiles $\left( PQ \right)_j$, $j \in [k]$.

\subsection{Testing multivariate Poisson distributions} \label{sec:Poisson} 

We now turn to understanding the testing problem described above; the setup is as follows. 
We consider a Bayesian hypothesis testing problem with $k$ hypotheses. 
The random variable $H$ takes values in $[k]$ with prior given by $p$, i.e., $\p \left( H = j \right) = p_j$. 
We do not observe $H$, but instead we observe a draw from a multivariate Poisson distribution whose mean depends on the realization of $H$: 
given $H = j$, the mean is $\lambda(j) \in \R_{+}^{k}$. 
In short:
\[
 D \, | \, H = j \sim \Poi \left( \lambda(j) \right), \qquad j \in [k].
\]
In more detail: 
\[
\p \left( D = d \, \middle| \, H = j \right) = \cP_{\lambda(j)} \left( d \right), \qquad d \in \Z_{+}^{k}, 
\]
where 
\[
 \cP_{\lambda(j)} \left( d \right) = \prod_{i \in [k]} \cP_{\lambda_{i}(j)} \left( d_{i} \right)
\]
and 
\[
 \cP_{\lambda_{i}(j)} \left( d_{i} \right) = \frac{\lambda_{i}(j)^{d_{i}}}{d_{i}!} e^{-\lambda_{i}(j)}.
\]
Our goal is to infer the value of $H$ from a realization of $D$. 
The error probability is minimized by the maximum a posteriori (MAP) rule, which, upon observing $D = d$, selects 
\[
 \argmax_{j \in [k]} \p \left( D = d \, \middle| \, H = j \right) p_j
\]
as an estimate for the value of $H$, with ties broken arbitrarily. 
Let $P_e$ denote the error of the MAP estimator. 
One can think of the MAP estimator as a tournament of $k-1$ pairwise comparisons of the hypotheses: 
if 
$
\p \left( D = d \, \middle| \, H = i \right) p_i 
> 
\p \left( D = d \, \middle| \, H = j \right) p_j
$
then the MAP estimate is not $j$. 
The probability that one makes an error during such a comparison is exactly
\begin{equation}\label{eq:error_ij}
 P_{e} \left( i, j \right) := \sum_{x \in \Z_{+}^{k}} \min \left\{ \p \left( D = x \, \middle| \, H = i \right) p_i , \p \left( D = x \, \middle| \, H = j \right) p_j \right\}.
\end{equation}
For finite $k$, the error of the MAP estimator is on the same order as the largest pairwise comparison error, i.e., $\max_{i,j} P_{e} \left( i, j \right)$. 
In particular, we have that 
\begin{equation}\label{eq:P_e}
\frac{1}{k-1} \sum_{i < j} P_{e} \left( i, j \right) 
\leq 
P_{e} 
\leq 
\sum_{i < j} P_{e} \left( i, j \right). 
\end{equation}
\begin{exercise}\label{ex:P_e_easy}
 Show~\eqref{eq:P_e}.
\end{exercise}
Thus we desire to understand the magnitude of the error probability $P_{e} \left( i, j \right)$ in~\eqref{eq:error_ij} in the particular case when the conditional distribution of $D$ given $H$ is a multivariate Poisson distribution with mean vector on the order of $\ln \left( n \right)$. The following result determines this error up to first order in the exponent. 
\begin{lemma}[Abbe and Sandon~\cite{AbbeSandon15}]\label{lem:D_+}
 For any $c_{1}, c_{2} \in \left( 0, \infty \right)^{k}$ with $c_1 \neq c_2$ and $p_1, p_2 > 0$, we have 
\begin{align}
\sum_{x \in \Z_{+}^{k}} \min \left\{ \cP_{\ln \left( n \right) c_{1}} \left( x \right) p_{1}, \cP_{\ln \left( n \right) c_{2}} \left( x \right) p_{2} \right\} &= O \left( n^{-D_{+} \left( c_1, c_2 \right) - \tfrac{\ln \ln \left( n \right)}{2 \ln \left( n \right)}} \right), \label{eq:UB}\\
\sum_{x \in \Z_{+}^{k}} \min \left\{ \cP_{\ln \left( n \right) c_{1}} \left( x \right) p_{1}, \cP_{\ln \left( n \right) c_{2}} \left( x \right) p_{2} \right\} &= \Omega \left( n^{-D_{+} \left( c_1, c_2 \right) - \tfrac{k \ln \ln \left( n \right)}{2 \ln \left( n \right)}} \right), \label{eq:LB}
\end{align}
where 
\begin{equation}\label{eq:D_+}
D_{+} \left( c_{1}, c_{2} \right) 
= \max_{t \in \left[ 0, 1 \right]} \sum_{i \in \left[ k \right]} \left( t c_{1} \left( i \right) + \left( 1 - t \right) c_{2} \left( i \right) - c_{1} \left( i \right)^{t} c_{2} \left( i \right)^{1-t} \right).
\end{equation}
\end{lemma} 
\begin{figure}[b!]
 \centering
 \includegraphics[width=0.55\textwidth]{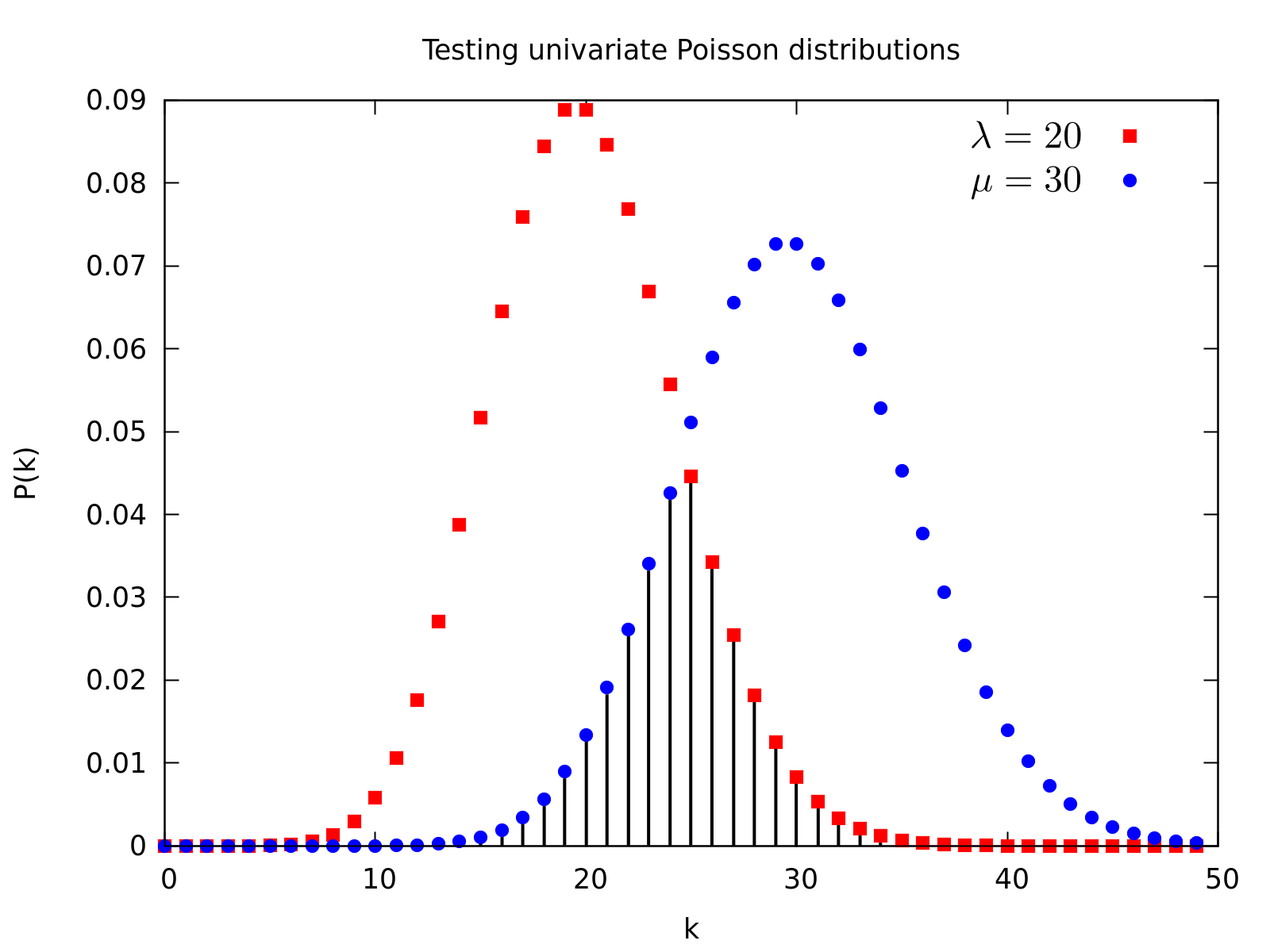}
 \caption{Testing univariate Poisson distributions. The figure plots the probability mass function of two Poisson distributions, with means $\lambda = 20$ and $\mu = 30$, respectively.}
 \label{fig:poisson}
\end{figure}
We do not go over the proof of this statement---which we leave to the reader as a challenging exercise---but we provide some intuition in the univariate case. 
Figure~\ref{fig:poisson} illustrates the probability mass function of two Poisson distributions, with means $\lambda = 20$ and $\mu = 30$, respectively. 
Observe that 
$\min \left\{ \mathcal{P}_{\lambda} \left( x \right), \mathcal{P}_{\mu} \left( x \right) \right\}$ 
decays rapidly away from 
$x_{\mathrm{max}} := \argmax_{x \in \mathbb{Z}_{+}} \min \left\{ \mathcal{P}_{\lambda} \left( x \right), \mathcal{P}_{\mu} \left( x \right) \right\}$,
so we can obtain a good estimate of the sum
$\sum_{x \in \mathbb{Z}_{+}} \min \left\{ \mathcal{P}_{\lambda} \left( x \right), \mathcal{P}_{\mu} \left( x \right) \right\}$ 
by simply estimating the term
$\min \left\{ \mathcal{P}_{\lambda} \left( x_{\mathrm{max}} \right), \mathcal{P}_{\mu} \left( x_{\mathrm{max}} \right) \right\}$. 
Now observe that $x_{\mathrm{max}}$ must satisfy 
$\mathcal{P}_{\lambda} \left( x_{\mathrm{max}} \right) \approx \mathcal{P}_{\mu} \left( x_{\mathrm{max}} \right)$; 
after some algebra this is equivalent to 
$x_{\mathrm{max}} \approx \frac{\lambda - \mu}{\log \left( \lambda / \mu \right)}$. 
Let $t^{*}$ denote the maximizer in the expression of $D_{+} \left( \lambda, \mu \right)$ in~\eqref{eq:D_+}. 
By differentiating in $t$, we obtain that $t^{*}$ satisfies 
$\lambda - \mu - \log \left( \lambda / \mu \right) \cdot \lambda^{t^{*}} \mu^{1 - t^{*}} = 0$, 
and so 
$\lambda^{t^{*}} \mu^{1 - t^{*}} = \frac{\lambda - \mu}{\log \left( \lambda / \mu \right)}$. 
Thus we see that 
$x_{\mathrm{max}} \approx \lambda^{t^{*}} \mu^{1 - t^{*}}$, 
from which, after some algebra, we get that 
$\mathcal{P}_{\lambda} \left( x_{\mathrm{max}} \right) \approx \mathcal{P}_{\mu} \left( x_{\mathrm{max}} \right) 
\approx \exp \left( - D_{+} \left( \lambda, \mu \right) \right)$.

The proof of~\eqref{eq:LB} in the multivariate case follows along the same lines: 
the single term corresponding to 
$x_{\mathrm{max}} := \argmax_{x \in \mathbb{Z}_{+}^{k}} \min \left\{ \mathcal{P}_{\ln \left( n \right) c_{1}} \left( x \right), \mathcal{P}_{\ln \left( n \right) c_{2}} \left( x \right) \right\}$ 
gives the lower bound. 
For the upper bound of~\eqref{eq:UB} one has to show that the other terms do not contribute much more. 
\begin{exercise}
 Prove Lemma~\ref{lem:D_+}.
\end{exercise}
Our conclusion is thus that the error exponent in testing multivariate Poisson distributions is given by the explicit quantity $D_{+}$ in~\eqref{eq:D_+}. 
The discussion in Section~\ref{sec:exact} then implies that $D_{+}$ plays an important role in the threshold for exact recovery. 
In particular, it intuitively follows from Lemma~\ref{lem:D_+} that a necessary condition for exact recovery should be that
\[
\min_{i,j \in [k], i \neq j} D_{+} \left( \left( PQ \right)_{i}, \left( PQ \right)_{j} \right) \geq 1.
\]
Suppose on the contrary that 
$D_{+} \left( \left( PQ \right)_{i}, \left( PQ \right)_{j} \right) < 1$ 
for some $i$ and $j$. 
This implies that the error probability in the testing problem is $\Omega \left( n^{\eps - 1} \right)$ for some $\eps > 0$ 
for all vertices in communities $i$ and $j$. 
Since the number of vertices in these communities is linear in $n$, 
and most of the hypothesis testing problems are approximately independent, 
one expects there to be no error in the testing problems 
with probability at most 
$\left( 1 - \Omega \left( n^{\eps - 1} \right) \right)^{\Omega \left( n \right)} = \exp\left( - \Omega \left( n^{\eps} \right) \right) = o(1)$.

\subsection{Chernoff-Hellinger divergence} \label{sec:CHdiv} 

Before moving on to the threshold for exact recovery in the general SBM, we discuss connections of $D_{+}$ to other, well-known measures of divergence. 
Writing 
\[
D_{t} \left( \mu, \nu \right) := \sum_{x \in \left[ k \right]} \left( t \mu \left( x \right) + \left( 1 - t \right) \nu \left( x \right) - \mu\left( x \right)^{t} \nu\left( x \right)^{1-t} \right)
\]
we have that 
\[
 D_{+} \left( \mu, \nu \right) = \max_{t \in \left[ 0, 1 \right]} D_{t} \left( \mu, \nu \right).
\]
For any fixed $t$, $D_{t}$ can be written as
\[
 D_{t} \left( \mu, \nu \right) = \sum_{x \in \left[ k \right]} \nu \left( x \right) f_{t} \left( \frac{\mu\left( x \right)}{\nu \left( x \right)} \right),
\]
where $f_{t} \left( x \right) = 1 - t + t x - x^{t}$, which is a convex function. 
Thus $D_{t}$ is an \emph{$f$-divergence}, part of a family of divergences that generalize the Kullback-Leibler (KL) divergence (also known as relative entropy), which is obtained for $f(x) = x \ln(x)$. 
The family of $f$-divergences with convex $f$ share many useful properties, and hence have been widely studied in information theory and statistics. 
The special case of 
$D_{1/2} \left( \mu, \nu \right) = \tfrac{1}{2} \left\| \sqrt{\mu} - \sqrt{\nu} \right\|_{2}^{2}$ 
is known as the Hellinger divergence. 
The Chernoff divergence is defined as 
$C_{*} \left( \mu, \nu \right) 
= \max_{t \in (0,1)} - \log \sum_{x} \mu(x)^{t} \nu(x)^{1-t}$, 
and so if $\mu$ and $\nu$ are probability vectors, then 
$D_{+} \left( \mu, \nu \right) = 1 - e^{-C_{*} \left( \mu, \nu \right)}$. 
Because of these connections, Abbe and Sandon termed $D_{+}$ the \emph{Chernoff-Hellinger divergence}.

While the quantity $D_{+}$ still might seem mysterious, even in light of these connections, 
a useful point of view is that Lemma~\ref{lem:D_+} gives $D_{+}$ an \emph{operational meaning}.

\subsection{Characterizing exact recoverability using CH-divergence} \label{sec:exact_CHdiv} 

Going back to the exact recovery problem in the general SBM, let us jump right in and state the recoverability threshold of Abbe and Sandon:
exact recovery in $\SBM(n,p, \ln(n) Q / n)$ is possible if and only if the CH-divergence between all pairs of community profiles is at least $1$. 

\begin{theorem}[Abbe and Sandon~\cite{AbbeSandon15}]\label{thm:exact}
Let $k \in \Z_{+}$ denote the number of communities, 
let $p \in (0,1)^{k}$ with $\left\| p \right\|_{1} = 1$ denote the community prior, 
let $P = \diag(p)$, 
and let $Q \in \left( 0, \infty \right)^{k \times k}$ be a symmetric $k \times k$ matrix with no two rows equal. 
Exact recovery is solvable in $\SBM \left( n, p, \ln(n) Q / n \right)$ if and only if 
\begin{equation}\label{eq:exact_condition}
\min_{i,j \in [k], i \neq j} D_{+} \left( \left( PQ \right)_{i}, \left( PQ \right)_{j} \right) \geq 1. 
\end{equation}
\end{theorem}

This theorem thus provides an operational meaning to the CH-divergence for the community recovery problem.

\begin{example}[Symmetric communities]
Consider again $k$ symmetric communities, that is, 
$p_i = 1/k$ for all $i \in [k]$, 
$Q_{i,j} = a$ if $i = j$, 
and $Q_{i,j} = b$ otherwise, with $a, b > 0$. 
Then exact recovery is solvable in $\SBM \left( n, p, \ln(n) Q / n \right)$ if and only if 
\begin{equation}\label{eq:exact_symm}
 \left| \sqrt{a} - \sqrt{b} \right| \geq \sqrt{k}. 
\end{equation}
We note that in this case $D_{+}$ is the same as the Hellinger divergence. 
\end{example}
\begin{exercise}\label{ex:symm_threshold}
Deduce from Theorem~\ref{thm:exact} that~\eqref{eq:exact_symm} gives the threshold in the example above. 
\end{exercise}

\subsubsection{Achievability}

Let us now see how Theorem~\ref{thm:exact} follows from the hypothesis testing results, starting with the achievability. 
When the condition~\eqref{eq:exact_condition} holds, then Lemma~\ref{lem:D_+} tells us that in the hypothesis testing problem between Poisson distributions the error of the MAP estimate is $o(1/n)$. 
Thus if the setting of the thought experiment described in Section~\ref{sec:exact} applies to every vertex, then by looking at the degree profiles of the vertices we can correctly reclassify all vertices, and the probability that we make an error is $o(1)$ by a union bound. 
However, the setting of the thought experiment does not quite apply. Nonetheless, in this logarithmic degree regime it is possible to partially reconstruct the labels of the vertices, with only $o(n)$ vertices being misclassified. The details of this partial reconstruction procedure would require a separate lecture---in brief, it determines whether two vertices are in the same community or not by looking at how their $\log(n)$ size neighborhoods interact---so now we will take this for granted.

It is possible to show that there exists a constant $\delta$ such that if one estimates the label of a vertex $v$ based on classifications of its neighbors that are wrong with probability $x$, 
then the probability of misclassifying $v$ is at most $n^{\delta x}$ times the probability of error if all the neighbors of $v$ were classified correctly. 
The issue is that the standard partial recovery algorithm has a constant error rate for the classifications, thus the error rate of the degree profiling step could be $n^{c}$ times as large as the error in the hypothesis testing problem, for some $c > 0$. This is an issue when $\min_{i\neq j} D_{+} \left( \left( PQ \right)_{i}, \left( PQ \right)_{j} \right) < 1 + c$.

To get around this, one can do multiple rounds of more accurate classifications. 
First, one obtains a partial reconstruction of the labels with an error rate that is a sufficiently low constant. 
After applying the degree-profiling step to each vertex, the classification error at each vertex is now $O(n^{-c'})$ for some $c' > 0$. 
Hence after applying another degree-profiling step to each vertex, 
the classification error at each vertex will now be at most $n^{\delta \times O( n^{-c'} )} \times o(1/n) = o(1/n)$. 
Thus applying a union bound at this stage we can conclude that all vertices are correctly labelled whp.

\subsubsection{Impossibility}

The necessity of condition~\eqref{eq:exact_condition} was already described at a high level at the end of Section~\ref{sec:Poisson}. 
Here we give some details on how to deal with the dependencies that arise.

Assume that~\eqref{eq:exact_condition} does not hold, and let $i$ and $j$ be two communities that violate the condition, i.e., for which 
$D_{+} \left( \left( PQ \right)_{i}, \left( PQ \right)_{j} \right) < 1$. 
We want to argue that vertices in communities $i$ and $j$ cannot all be distinguished, and so any classification algorithm has to make at least one error whp. 
An important fact that we use is that the lower bound~\eqref{eq:LB} arises from a particular choice of degree profile that is both likely for the two communities. 
Namely, define the degree profile $x$ by 
\[
 x_{\ell} = \left\lfloor \left( PQ \right)_{\ell,i}^{t^{*}} \left( PQ \right)_{\ell,j}^{1-t^{*}} \ln \left( n \right) \right\rfloor
\]
for every $\ell \in [k]$, where $t^{*} \in \left[ 0, 1 \right]$ is the maximizer in $D_{+} \left( \left( PQ \right)_{i}, \left( PQ \right)_{j} \right)$, i.e., the value for which 
$D_{+} \left( \left( PQ \right)_{i}, \left( PQ \right)_{j} \right) = D_{t^{*}} \left( \left( PQ \right)_{i}, \left( PQ \right)_{j} \right)$. 
Then Lemma~\ref{lem:D_+} tells us that for any vertex in community $i$ or $j$, the probability that it has degree profile $x$ is at least 
\[
 \Omega \left( n^{- D_{+} \left( \left( PQ \right)_{i}, \left( PQ \right)_{j} \right)} / \left( \ln \left( n \right) \right)^{k/2} \right),
\]
which is at least $\Omega\left( n^{\eps - 1} \right)$ for some $\eps > 0$ by assumption.

To show that this holds for many vertices in communities $i$ and $j$ at once, 
we first select a random set $S$ of $n / \left( \ln \left( n \right) \right)^3$ vertices. 
Whp the intersection of $S$ with any community $\ell$ is within $\sqrt{n}$ of the expected value $p_{\ell} n / \left( \ln \left( n \right) \right)^3$, 
and furthermore a randomly selected vertex in $S$ is not connected to any other vertex in $S$. 
Thus the distribution of a vertex's degree profile excluding connections to vertices in $S$ is essentially a multivariate Poisson distribution as before. 
We call a vertex in $S$ \emph{ambiguous} if for each $\ell \in [k]$ it has exactly $x_{\ell}$ neighbors in community $\ell$ that are not in $S$. 
By Lemma~\ref{lem:D_+} we have that a vertex in $S$ that is in community $i$ or $j$ is ambiguous with probability $\Omega \left( n^{\eps - 1} \right)$. 
By definition, for a fixed community assignment and choice of $S$, there is no dependence on whether two vertices are ambiguous. 
Furthermore, due to the choice of the size of $S$, whp there are 
at least $\ln \left( n \right)$ ambiguous vertices in community $i$ and 
at least $\ln \left( n \right)$ ambiguous vertices in community $j$ 
that are not adjacent to any other vertices in $S$. 
These $2 \ln \left( n \right)$ are indistinguishable, so no algorithm classifies all of them correctly with probability greater than 
$1 / \binom{2\ln \left( n \right)}{\ln \left( n \right)}$, which tends to $0$ as $n \to \infty$.

\subsubsection{The finest exact partition recoverable}

We conclude by mentioning that this threshold generalizes to finer questions. 
If exact recovery is not possible, what is the finest partition that can be recovered? 
We say that exact recovery is solvable for a community partition 
$\left[ k \right] = \sqcup_{s=1}^{t} A_s$, where $A_s$ is a subset of $[k]$, 
if there exists an algorithm that whp assigns to every vertex an element of $\left\{ A_{1}, \dots, A_{t} \right\}$ that contains its true community. 
The finest partition that is exactly recoverable can also be expressed using CH-divergence in a similar fashion. 
It is the largest collection of disjoint subsets such that the CH-divergence between these subsets is at least $1$, 
where the CH-divergence between two subsets is defined as the minimum of the CH-divergences between any two community profiles in these subsets. 
\begin{theorem}[Abbe and Sandon~\cite{AbbeSandon15}]\label{thm:exact_fine}
Under the same settings as in Theorem~\ref{thm:exact}, exact recovery is solvable in $\SBM \left( n, p, \ln(n) Q / n \right)$ for a partition $\left[ k \right] = \sqcup_{s=1}^{t} A_s$ if and only if 
\begin{equation*}\label{eq:exact_condition_fine}
D_{+} \left( \left( PQ \right)_{i}, \left( PQ \right)_{j} \right) \geq 1 
\end{equation*}
for every $i$ and $j$ in different subsets of the partition. 
\end{theorem}

%

\newpage

\section{Lecture 2: Estimating the dimension of a random geometric graph on a high-dimensional sphere} \label{sec:lec2} 

Many real-world networks have strong structural features and our goal is often to recover these hidden structures. 
In the previous lecture we studied the fundamental limits of inferring communities in the stochastic block model, a natural generative model for graphs with community structure. 
Another possibility is \emph{geometric structure}. 
Many networks coming from physical considerations naturally have an underlying geometry, such as the network of major roads in a country. 
In other networks this stems from a latent feature space of the nodes. 
For instance, in social networks a person might be represented by a feature vector of their interests, and two people are connected if their interests are close enough; 
this latent metric space is referred to as the \emph{social space}~\cite{HRH02}.

In such networks the natural questions probe the underlying geometry. 
Can one detect the presence of geometry? 
If so, can one estimate various aspects of the geometry, e.g., an appropriately defined dimension? 
In this lecture we study these questions in a particularly natural and simple generative model of a random geometric graph: 
$n$ points are picked uniformly at random on the $d$-dimensional sphere, and two points are connected by an edge if and only if they are sufficently close.\footnote{This lecture is based on~\cite{BDER16}.}

We are particularly interested in the \emph{high-dimensional} regime, 
motivated by recent advances in all areas of applied mathematics, and in particular statistics and learning theory, 
where high-dimensional feature spaces are becoming the new norm.  
While the low-dimensional regime has been studied for a long time in probability theory~\cite{Penrosebook}, 
the high-dimensional regime brings about a host of new and interesting questions.

\subsection{A simple random geometric graph model and basic questions} \label{sec:model} 

Let us now define more precisely the random geometric graph model we consider and the questions we study. 
In general, a geometric graph is such that each vertex is labeled with a point in some metric space, 
and an edge is present between two vertices if the distance between the corresponding labels is smaller than some prespecified threshold. 
We focus on the case where the underlying metric space is the Euclidean sphere $\mathbb{S}^{d-1}=\left\{ x\in \R^d: \|x\|_{2} = 1 \right\}$, 
and the latent labels are i.i.d.\ uniform random vectors in $\mathbb{S}^{d-1}$. 
We denote this model by $G(n,p,d)$, where $n$ is the number of vertices and $p$ is the probability of an edge between two vertices ($p$ determines the threshold distance for connection). 
This model is closely related to latent space approaches to social network analysis~\cite{HRH02}.

Slightly more formally, $G(n,p,d)$ is defined as follows. 
Let $X_1,\ldots,X_n$ be independent random vectors, uniformly distributed on $\mathbb{S}^{d-1}$. 
In $G(n,p,d)$, distinct vertices $i \in [n]$ and $j \in [n]$ are connected by an edge if and only if $\langle X_i,X_j \rangle \ge t_{p,d}$, 
where the threshold value $t_{p,d} \in [-1,1]$ is such that $\P\left( \langle X_1,X_2 \rangle \ge t_{p,d} \right) = p$. 
For example, when $p = 1/2$ we have $t_{p,d} = 0$. 

The most natural random graph model without any structure is the standard Erd\H{o}s-R\'enyi random graph $G(n,p)$, 
where any two of the $n$ vertices are independently connected with probability~$p$.

We can thus formalize the question of detecting underlying geometry as a simple hypothesis testing question. 
The null hypothesis is that the graph is drawn from the Erd\H{o}s-R\'enyi model, 
while the alternative is that it is drawn from $G(n,p,d)$. In brief: 
\begin{equation}\label{eq:hypothesis_Gnpd}
 H_{0} : G \sim G(n,p), 
\qquad \qquad 
 H_{1} : G \sim G(n,p,d).
\end{equation}
To understand this question, the basic quantity we need to study is the total variation distance between the two distributions on graphs, $G(n,p)$ and $G(n,p,d)$, denoted by $\TV \left( G(n,p), G(n,p,d) \right)$; 
recall that the total variation distance between two probability measures $P$ and $Q$ is defined as 
$\TV \left( P, Q \right) = \tfrac{1}{2} \left\| P - Q \right\|_{1} = \sup_{A} \left| P(A) - Q(A) \right|$.  
We are interested in particular in the case when the dimension $d$ is \emph{large}, growing with $n$. 

It is intuitively clear that if the geometry is too high-dimensional, then it is impossible to detect it, 
while a low-dimensional geometry will have a strong effect on the generated graph and will be detectable. 
How fast can the dimension grow with $n$ while still being able to detect it? 
Most of this lecture will focus on this question.

If we can detect geometry, then it is natural to ask for more information. 
Perhaps the ultimate goal would be to find an embedding of the vertices into an appropriate dimensional sphere that is a \emph{true representation}, 
in the sense that the geometric graph formed from the embedded points is indeed the original graph. 
More modestly, can the dimension be estimated? We touch on this question at the end of the lecture.

\subsection{The dimension threshold for detecting underlying geometry} \label{sec:results} 

The high-dimensional setting of the random geometric graph $G(n,p,d)$ was first studied by Devroye, Gy\"orgy, Lugosi, and Udina~\cite{DGLU11}, who showed that if $n$ is fixed and $d \to \infty$, then 
\[
 \TV \left( G(n,p), G(n,p,d) \right) \to 0,
\]
that is, geometry is indeed lost in high dimensions. 
More precisely, they show that this convergence happens when $d \gg n^{7} 2^{n^2 / 2}$.\footnote{Throughout these notes we use standard asymptotic notation; 
for instance, $f\left( t \right) \ll g\left( t \right)$ as $t\to\infty$ if $\lim_{t\to\infty} f\left( t \right) / g \left( t \right) = 0$.} 
This follows by observing that for fixed $n$, the multivariate central limit theorem implies that as $d \to \infty$, the inner products of the latent vectors converge in distribution to a standard Gaussian:
\[
 \left( \frac{1}{\sqrt{d}} \left\langle X_i, X_j \right\rangle \right)_{\left\{i,j\right\} \in \binom{\left[n\right]}{2}} \stackrel{d\to\infty}{\Longrightarrow} \cN \left( 0, I_{\binom{n}{2}} \right).
\]
The Berry-Esseen theorem gives a convergence rate, which then allows to show that for any graph $G$ on $n$ vertices, 
$\left| \p \left( G(n,p) = G \right)  - \p \left( G(n,p,d) = G \right) \right| = O \left( \sqrt{n^{7} / d} \right)$; 
the factor of $2^{n^2 / 2}$ comes from applying this bound to every term in the $L_1$ distance.

However, the result above is not tight, and we seek to understand the fundamental limits to detecting underlying geometry. 
The dimension threshold for dense graphs was recently found in~\cite{BDER16}, and it turns out that it is $d \approx n^3$, in the following sense. 

\begin{theorem}[Bubeck, Ding, Eldan, R\'acz~\cite{BDER16}]\label{thm:gnpd_n3}
Let $p \in (0,1)$ be fixed. Then
\begin{numcases}
 {\TV \left( G(n,p), G(n,p,d) \right) \to }
 0, & $\text{ if } d \gg n^3,$ \label{eq:impossible} \\ 
 1, & $\text{ if } d \ll n^3.$ \label{eq:possible}
\end{numcases}
Moreover, in the latter case there exists a computationally efficient test to detect underlying geometry (with running time $O\left( n^{3} \right)$).
\end{theorem}
Most of the lecture will be devoted to understanding this theorem. 
At the end we will consider this same question for \emph{sparse graphs} (where $p = c/n$), where determining the dimension threshold is an intriguing open problem.

\subsection{The triangle test} \label{sec:triangle} 

A natural test to uncover geometric structure is to count the number of triangles in $G$. 
Indeed, in a purely random scenario, vertex $u$ being connected to both $v$ and $w$ says nothing about whether $v$ and $w$ are connected. 
On the other hand, in a geometric setting this implies that $v$ and $w$ are close to each other due to the triangle inequality, thus increasing the probability of a connection between them.  
This, in turn, implies that the expected number of triangles is larger in the geometric setting, given the same edge density. 
Let us now compute what this statistic gives us. 
\begin{figure}[h!]
 \centering
 \includegraphics[width=0.15\textwidth]{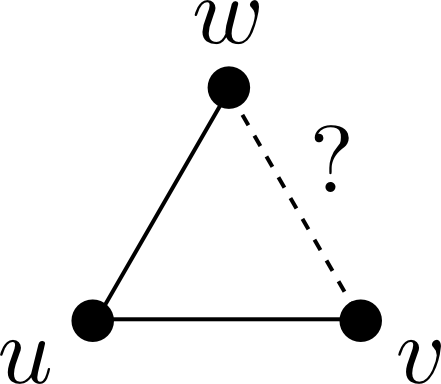}
 \caption{Given that $u$ is connected to both $v$ and $w$, $v$ and $w$ are more likely to be connected under $G(n,p,d)$ than under $G(n,p)$.}
 \label{fig:triangle}
\end{figure}

For a graph $G$, let $A$ denote its adjacency matrix, i.e., $A_{i,j} = 1$ if vertices $i$ and $j$ are connected, and $0$ otherwise. 
Then $T_{G} \left( i,j,k \right) := A_{i,j} A_{i,k} A_{j,k}$ is the indicator variable that three vertices $i$, $j$, and $k$ form a triangle, and so the number of triangles in $G$ is 
\begin{equation*}\label{eq:triangle_def}
 T(G) := \sum_{\{i,j,k\} \in \binom{[n]}{3}} T_{G} \left( i,j,k \right).
\end{equation*}
By linearity of expectation, for both models the expected number of triangles is $\binom{n}{3}$ times the probability of a triangle between three specific vertices. 
For the Erd\H{o}s-R\'enyi random graph the edges are independent, so the probability of a triangle is $p^3$, and thus we have 
\[
 \E \left[ T \left( G(n,p) \right) \right] = \binom{n}{3} p^3. 
\]

For $G(n,p,d)$ it turns out that for any fixed $p \in \left( 0, 1 \right)$ we have 
\begin{equation}\label{eq:triangle_gnpd}
 \p \left( T_{G(n,p,d)} \left( 1, 2, 3 \right) = 1 \right) \approx p^{3} \left( 1 + \frac{C_{p}}{\sqrt{d}} \right)
\end{equation}
for some constant $C_{p} > 0$, which gives that 
\[
 \E \left[ T \left( G(n,p,d) \right) \right] \geq \binom{n}{3} p^3 \left( 1 + \frac{C_{p}}{\sqrt{d}} \right). 
\]
Showing~\eqref{eq:triangle_gnpd} is somewhat involved, but in essence it follows from the \emph{concentration of measure} phenomenon on the sphere, 
namely that most of the mass on the high-dimensional sphere is located in a band of $O \left( 1 / \sqrt{d} \right)$ around the equator. 
We sketch here the main intuition for $p=1/2$, which is illustrated in Figure~\ref{fig:triangle_gnpd}. 

Let $X_1$, $X_2$, and $X_3$ be independent uniformly distributed points in $\mathbb{S}^{d-1}$. Then 
\begin{multline*}
 \p \left( T_{G(n,1/2,d)} \left( 1, 2, 3 \right) = 1 \right) \\
\begin{aligned}
&= \p \left( \langle X_1, X_2 \rangle \geq 0, \langle X_1, X_3 \rangle \geq 0, \langle X_2, X_3 \rangle \geq 0 \right) \\
&= \p \left( \langle X_2, X_3 \rangle \geq 0 \, \middle| \, \langle X_1, X_2 \rangle \geq 0, \langle X_1, X_3 \rangle \geq 0 \right) \p \left( \langle X_1, X_2 \rangle \geq 0, \langle X_1, X_3 \rangle \geq 0 \right) \\
&= \frac{1}{4} \times \p \left( \langle X_2, X_3 \rangle \geq 0 \, \middle| \, \langle X_1, X_2 \rangle \geq 0, \langle X_1, X_3 \rangle \geq 0 \right),
\end{aligned}
\end{multline*}
where the last equality follows by independence. 
So what remains is to show that this latter conditional probability is approximately $1/2 + c / \sqrt{d}$. 
To compute this conditional probability what we really need to know is the typical angle is between $X_1$ and $X_2$. 
By rotational invariance we may assume that $X_1 = (1,0,0, \dots, 0)$, and hence $\langle X_1, X_2 \rangle = X_{2} (1)$, the first coordinate of $X_{2}$. 
One way to generate $X_2$ is to sample a $d$-dimensional standard Gaussian and then normalize it by its length. 
Since the norm of a $d$-dimensional standard Gaussian is very well concentrated around $\sqrt{d}$, it follows that $X_{2}(1)$ is on the order of $1/\sqrt{d}$. 
Conditioned on $X_{2}(1) \geq 0$, this typical angle gives the boost in the conditional probability that we see. See Figure~\ref{fig:triangle_gnpd} for an illustration.

\begin{figure}[h!]
 \centering
 \includegraphics[width=0.55\textwidth]{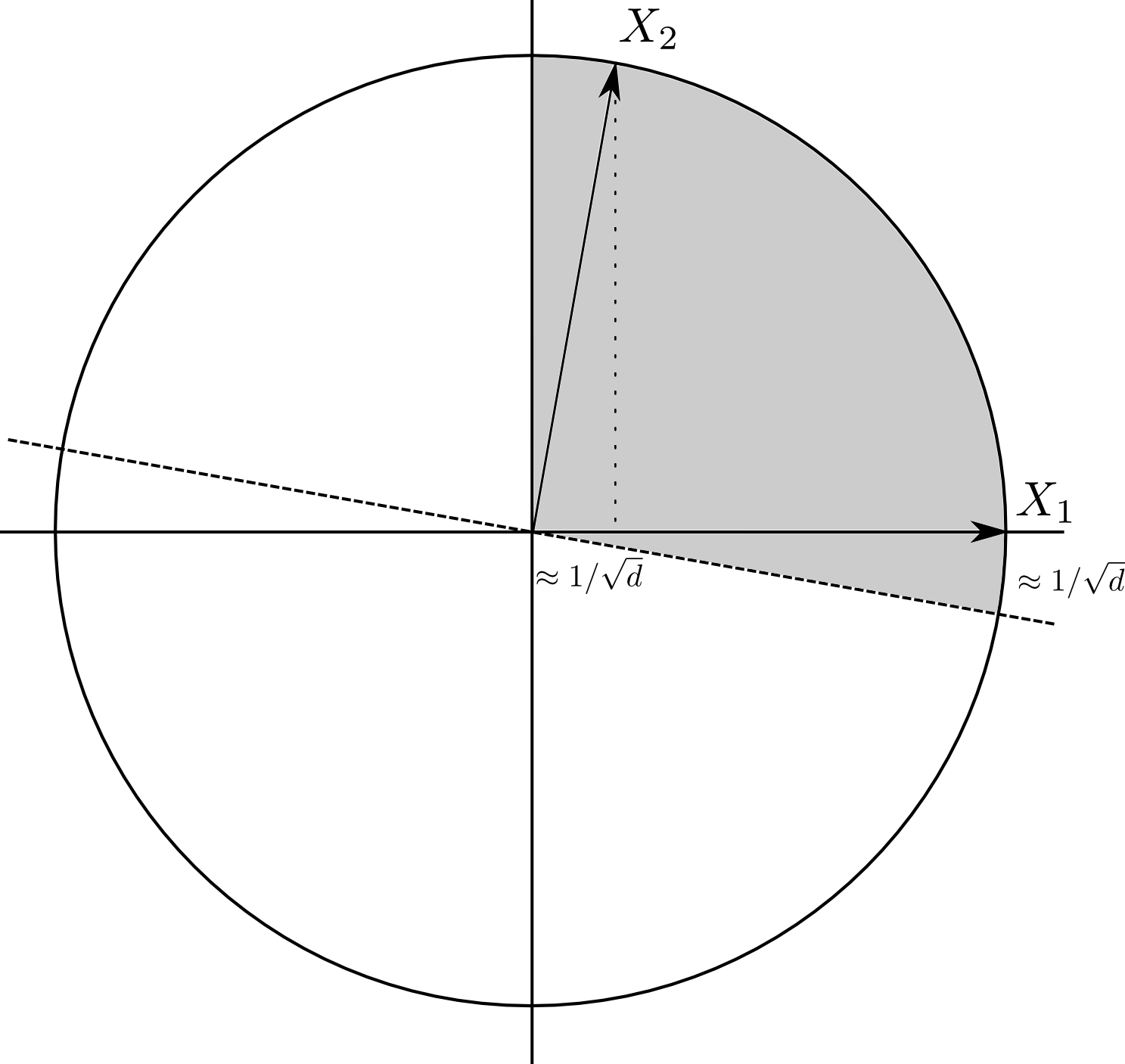}
 \caption{If $X_{1}$ and $X_{2}$ are two independent uniform points on the $d$-dimensional sphere $\mathbb{S}^{d-1}$, then their inner product $\left\langle X_1, X_2 \right\rangle$ is on the order of $1/\sqrt{d}$ due to the concentration of measure phenomenon on the sphere. This then implies that the probability of a triangle in $G(n,1/2,d)$ is $(1/2)^3 + c/\sqrt{d}$ for some constant $c > 0$.}
 \label{fig:triangle_gnpd}
\end{figure}

Thus we see that the boost in the number of triangles in the geometric setting is $\Theta \left( n^{3} / \sqrt{d} \right)$ in expectation:
\[
 \E \left[ T \left( G(n,p,d) \right) \right] - \E \left[ T \left( G(n,p) \right) \right] \geq \binom{n}{3} \frac{C_p}{\sqrt{d}}.
\]
To be able to tell apart the two graph distributions based on the number of triangles, the boost in expectation needs to be much greater than the standard deviation. 

\begin{exercise}\label{ex:var}
 Show that 
\[
 \Var \left( T \left( G \left( n, p \right) \right) \right) = \binom{n}{3} \left( p^{3} - p^{6} \right) + \binom{n}{4} \binom{4}{2} \left( p^{5} - p^{6} \right)
\]
and that 
$\Var \left( T \left( G \left( n, p, d \right) \right) \right) \leq n^4$.
\end{exercise}

\begin{exercise}\label{ex:chebyshev}
 Show that if 
\[
 \left| \E \left[ T \left( G(n,p,d) \right) \right] - \E \left[ T \left( G(n,p) \right) \right] \right| 
\gg 
\max \left\{ \sqrt{\Var \left( T \left( G \left( n, p \right) \right) \right)}, \sqrt{\Var \left( T \left( G \left( n, p, d \right) \right) \right)} \right\},
\]
then 
\[
  \TV \left( G(n,p), G(n,p,d) \right) \to 1.
\]
\end{exercise}

Putting together Exercises~\ref{ex:var} and~\ref{ex:chebyshev} we see that 
$\TV \left( G(n,p), G(n,p,d) \right) \to 1$ 
if 
$n^{3} / \sqrt{d} \gg \sqrt{n^4}$, 
which is equivalent to $d \ll n^2$.

\subsection{Signed triangles are more powerful} \label{sec:signed_triangle} 

While triangles detect geometry up until $d \ll n^2$, are there even more powerful statistics that detect geometry for larger dimensions? 
One can check that longer cycles also only work when $d \ll n^2$, as do several other natural statistics. 
Yet it turns out that the underlying geometry can be detected even when $d \ll n^{3}$.

The simple idea that leads to this improvement is to consider \emph{signed triangles}. 
We have already noticed that triangles are more likely in the geometric setting than in the purely random setting. 
This also means that induced wedges (i.e., when there are exactly two edges among the three possible ones) are less likely in the geometric setting. 
Similarly, induced single edges are more likely, and induced independent sets on three vertices are less likely in the geometric setting. 
Figure~\ref{fig:triangles_signed} summarizes these observations. 
\begin{figure}[h!]
 \centering
 \includegraphics[width=0.55\textwidth]{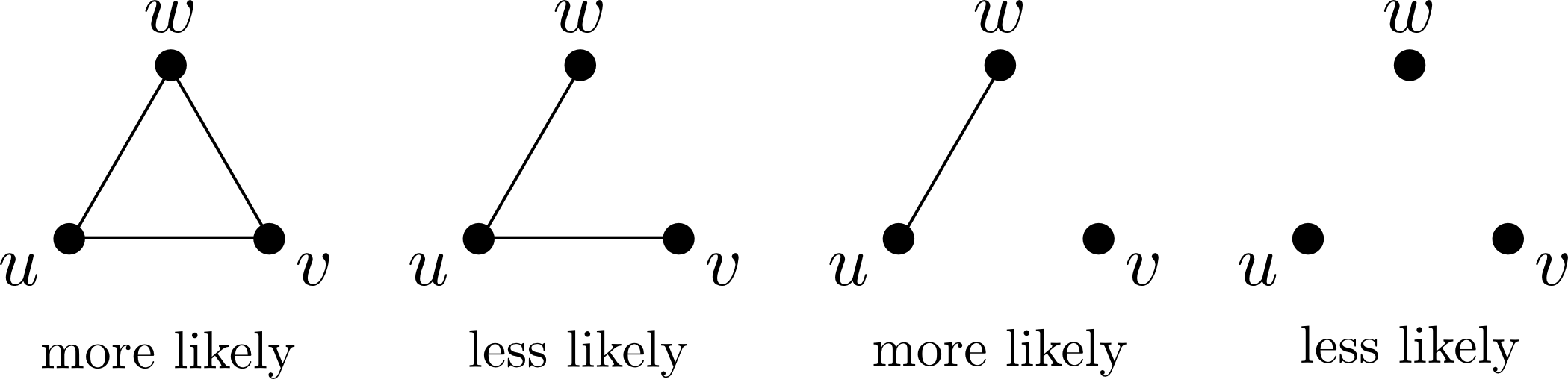}
 \caption{This figure summarizes which patterns are more or less likely in the geometric setting than in the purely random setting. The signed triangles statistic reweights the different patterns with positive and negative weights.}
 \label{fig:triangles_signed}
\end{figure}

The signed triangles statistic incorporates these observations by giving the different patterns positive or negative weights. More precisely, we define 
\[
 \tau \left( G \right) := \sum_{\{i,j,k\} \in \binom{[n]}{3}} \left( A_{i,j} - p \right) \left( A_{i,k} - p \right) \left( A_{j,k} - p \right).
\]
The key insight motivating this definition is that the variance of signed triangles is \emph{much smaller} than the variance of triangles, 
due to the cancellations introduced by the centering of the adjacency matrix: the $\Theta \left( n^{4} \right)$ term vanishes, leaving only the $\Theta \left( n^{3} \right)$ term. 
\begin{exercise}\label{ex:var_signed}
Show that 
\[
 \E \left[ \tau \left( G(n,p) \right) \right] = 0
\]
and 
\[
 \Var \left( \tau \left( G(n,p) \right) \right) = \binom{n}{3} p^{3} \left( 1 - p \right)^{3}.
\]
\end{exercise}
On the other hand it can be shown that 
\begin{equation}\label{eq:tau_gnpd}
\E \left[ \tau \left( G(n,p,d) \right) \right] 
\geq c_p n^{3} / \sqrt{d}, 
\end{equation}
so the gap between the expectations remains.  
Furthermore, it can also be shown that the variance also decreases for $G(n,p,d)$ and we have 
\begin{equation}\label{eq:tau_gnpd_var}
 \Var \left( \tau \left( G(n,p,d) \right) \right) \leq n^{3} + \frac{3 n^{4}}{d}.
\end{equation}
Putting everything together and using Exercise~\ref{ex:chebyshev} for the signed triangles statistic $\tau$, 
we get that 
$\TV \left( G(n,p), G(n,p,d) \right) \to 1$ 
if 
$n^{3} / \sqrt{d} \gg \sqrt{n^3 + n^{4} / d}$, 
which is equivalent to $d \ll n^3$. 
This concludes the proof of~\eqref{eq:possible} from Theorem~\ref{thm:gnpd_n3}.

\subsection{Barrier to detecting geometry: when Wishart becomes GOE} \label{sec:LB} 

We now turn to proving~\eqref{eq:impossible}, which, together with~\eqref{eq:possible}, 
shows that the threshold dimension for detecting geometry is $n^{3}$. 
This also shows that the signed triangle statistic is near-optimal, since it can detect geometry whenever $d \ll n^{3}$.

There are essentially three main ways to bound the total variation of two distributions from above: 
(i) if the distributions have nice formulas associated with them, then exact computation is possible; 
(ii) through \emph{coupling} the distributions; 
or (iii) by using inequalities between probability metrics to switch the problem to bounding a different notion of distance between the distributions. 
Here, while the distribution of $G(n,p,d)$ does not have a nice formula associated with it, 
the main idea is to view this random geometric graph as a function of an 
$n \times n$ Wishart matrix with $d$ degrees of freedom---i.e., a matrix of inner products of $n$ $d$-dimensional Gaussian vectors---denoted by $W(n,d)$. 
It turns out that one can view $G(n,p)$ as (essentially) the same function of an $n \times n$ GOE random matrix---i.e., a symmetric matrix with i.i.d.\ Gaussian entries on and above the diagonal---denoted by $M(n)$. 
The upside of this is that both of these random matrix ensembles have explicit densities that allow for explicit computation. 
We explain this connection here in the special case of $p = 1/2$ for simplicity; see~\cite{BDER16} for the case of general $p$.

Recall that if $Y_1$ is a standard normal random variable in $\R^d$, then $Y_1 / \left\| Y_1 \right\|$ is uniformly distributed on the sphere $\mathbb{S}^{d-1}$. 
Consequently we can view $G\left( n, 1/2, d \right)$ as a function of an appropriate Wishart matrix, as follows. 
Let $Y$ be an $n \times d$ matrix where the entries are i.i.d.\ standard normal random variables, and let $W \equiv W (n,d) = YY^T$  be the corresponding $n \times n$ Wishart matrix. 
Note that $W_{ii} = \left\langle Y_i, Y_i \right\rangle = \left\| Y_i \right\|^2$ and so 
$\left\langle Y_i / \left\| Y_i \right\|, Y_j / \left\| Y_j \right\| \right\rangle = W_{ij} / \sqrt{W_{ii} W_{jj}}$. 
Thus the $n \times n$ matrix $A$ defined as
\[
 A_{i,j} =
\begin{cases}
 1 & \text{if } W_{ij} \geq 0 \text{ and } i \neq j\\
 0 & \text{otherwise}
\end{cases}
\]
has the same law as the adjacency matrix of $G\left(n,1/2,d\right)$. 
Denote the map that takes $W$ to $A$ by $H$, i.e., $A = H \left( W \right)$.

In a similar way we can view $G \left( n, 1/2 \right)$ as a function of an $n \times n$ matrix drawn from the Gaussian Orthogonal Ensemble (GOE). 
Let $M\left( n \right)$ be a symmetric $n \times n$ random matrix where the diagonal entries are i.i.d.\ normal random variables with mean zero and variance 2, and the entries above the diagonal are i.i.d.\ standard normal random variables, with the entries on and above the diagonal all independent. 
Then $B = H \left( M(n) \right)$ has the same law as the adjacency matrix of $G(n,p)$. 
Note that $B$ only depends on the sign of the off-diagonal elements of $M\left(n \right)$, so in the definition of $B$ we can replace $M\left( n \right)$ with $M \left( n, d \right) := \sqrt{d} M \left( n \right) + d I_n$, where $I_n$ is the $n \times n$ identity matrix. 

We can thus conclude that 
\begin{align*}
 \TV \left( G(n,1/2,d), G(n,1/2) \right) 
&= \TV \left( H \left( W(n,d) \right), H \left( M(n,d) \right) \right) \\
&\leq \TV \left(  W(n,d),  M(n,d) \right).
\end{align*}
The densities of these two random matrix ensembles are well known. 
Let $\mathcal{P} \subset \R^{n^2}$ denote the cone of positive semidefinite matrices. 
When $d \geq n$, $W(n,d)$ has the following density with respect to the Lebesgue measure on $\mathcal{P}$:
\[
 f_{n,d} \left( A \right) := \frac{\left( \det \left( A \right) \right)^{\frac{1}{2} \left( d - n - 1 \right)} \exp \left( - \frac{1}{2} \Tr \left( A \right) \right)}{2^{\frac{1}{2}dn} \pi^{\frac{1}{4} n \left( n-1\right)} \prod_{i=1}^n \Gamma \left( \frac{1}{2} \left( d+1-i \right) \right)},
\]
where $\Tr\left( A \right)$ denotes the trace of the matrix $A$. 
It is also known that the density of a GOE random matrix with respect to the Lebesgue measure on $\R^{n^2}$ is
$A \mapsto \left( 2 \pi \right)^{-\frac{1}{4} n \left( n + 1 \right)} 2^{-\frac{n}{2}} \exp \left( - \frac{1}{4} \Tr \left( A^2 \right) \right)$ 
and so the density of $M \left( n, d \right)$ with respect to the Lebesgue measure on $\R^{n^2}$ is
\[
 g_{n,d} \left( A \right) := \frac{\exp \left( - \frac{1}{4d} \Tr \left( \left( A - d I_n \right)^2 \right) \right)}{\left( 2\pi d \right)^{\frac{1}{4} n \left( n + 1 \right)} 2^{\frac{n}{2}}}.
\]

These explicit formulas allow for explicit calculations. 
In particular, one can show that the log-ratio of the densities is $o(1)$ with probability $1-o(1)$ according to the measure induced by $M(n,d)$. 
This follows from writing out the Taylor expansion of the log-ratio of the densities and using known results about the empirical spectral distribution of Wigner matrices (in particular that it converges to a semi-circle law). 
The outcome of the calculation is the following result, proven independently and simultaneously by Bubeck~et~al.\ and Jiang and Li. 

\begin{theorem}[Bubeck, Ding, Eldan, R\'acz~\cite{BDER16}; Jiang, Li~\cite{jiang2013approximation}]\label{thm:Wishart_GOE}
Define the random matrix ensembles $W\left( n, d \right)$ and $M \left( n, d \right)$ as above. If $d / n^3 \to \infty$, then
\begin{equation*}\label{eq:Wishart_GOE}
  \TV \left( W \left( n, d \right), M \left( n, d \right) \right) \to 0.
\end{equation*}
\end{theorem}

We conclude that it is impossible to detect underlying geometry whenever $d \gg n^{3}$.

\subsection{Estimating the dimension} \label{sec:dimension} 

Until now we discussed \emph{detecting} geometry. 
However, the insights gained above allow us to also touch upon the more subtle problem of \emph{estimating} the underlying dimension $d$.

Dimension estimation can also be done by counting the ``number'' of signed triangles as in Section~\ref{sec:signed_triangle}. 
However, here it is necessary to have a bound on the difference of the expected number of signed triangles between consecutive dimensions; 
the lower bound of~\eqref{eq:tau_gnpd} is not enough. 
Still, we believe that the right hand side of~\eqref{eq:tau_gnpd} should give the true value of the expected value for an appropriate constant $c_p$, and hence we expect to have that 
\begin{equation}\label{eq:exp_gap_d}
 \E \left[ \tau \left( G(n,p,d) \right) \right] - \E \left[ \tau \left( G(n,p,d+1) \right) \right] = \Theta \left( \frac{n^{3}}{d^{3/2}} \right).
\end{equation}
Thus, using the variance bound in~\eqref{eq:tau_gnpd_var}, 
we get that dimension estimation should be possible using signed triangles whenever 
$n^{3} / d^{3/2} \gg \sqrt{n^3 + n^{4} / d}$, 
which is equivalent to $d \ll n$.

Showing~\eqref{eq:exp_gap_d} for general $p$ seems involved; Bubeck~et~al.\ showed that it holds for $p = 1/2$, which can be considered as a proof of concept. We thus have the following. 
\begin{theorem}[Bubeck, Ding, Eldan, R\'acz~\cite{BDER16}]\label{thm:dim_est}
 There exists a universal constant $C > 0$ such that for all integers $n$ and $d_1 < d_2$, one has
\[
 \TV \left( G (n, 1/2, d_1), G(n,1/2, d_2) \right) \geq 1 - C \left( \frac{d_1}{n} \right)^{2}.
\]
\end{theorem}
This result is tight, as demonstrated by a result of Eldan~\cite{Eldan11}, 
which states that when $d \gg n$, 
the Wishart matrices $W(n,d)$ and $W(n,d+1)$ are indistinguishable. 
By the discussion in Section~\ref{sec:LB}, this directly implies that 
$G(n,1/2,d)$ and $G(n,1/2,d+1)$ are indistinguishable. 
\begin{theorem}[Eldan~\cite{Eldan11}]\label{eq:Ronen_LB}
 There exists a universal constant $C > 0$ such that for all integers $n < d$, 
\[
\TV \left( G(n,1/2,d), G(n,1/2,d+1) \right) 
\leq \TV \left( W(n,d), W(n,d+1) \right) 
\leq C \sqrt{\left( \frac{d+1}{d-n} \right)^{2} - 1}.
\]
\end{theorem}

\subsection{The mysterious sparse regime} \label{sec:sparse} 

The discussion so far has focused on dense graphs, i.e., assuming $p \in (0,1)$ is constant, 
where Theorem~\ref{thm:gnpd_n3} tightly characterizes when the underlying geometry can be detected. 
The same questions are interesting for \emph{sparse graphs} as well, where the average degree is constant or slowly growing with $n$. 
However, since there are so few edges, this regime is much more challenging.

It is again natural to consider the number of triangles as a way to distinguish between $G(n,c/n)$ and $G(n, c/n,d)$. 
A calculation shows that this statistic works whenever 
$d \ll \log^{3} \left( n \right)$. 
\begin{theorem}[Bubeck, Ding, Eldan, R\'acz~\cite{BDER16}]\label{thm:sparse}
Let $c > 0$ be fixed and assume $d / \log^{3} \left( n \right) \to 0$. 
Then 
\[
\TV \left( G\left( n, \frac{c}{n} \right),  G\left( n, \frac{c}{n}, d \right) \right) \to 1.
\]
\end{theorem}
In contrast with the dense regime, in the sparse regime the signed triangle statistic $\tau$ does not give significantly more power than the triangle statistic $T$. 
This is because in the sparse regime, with high probability, 
the graph does not contain any $4$-vertex subgraph with at least $5$ edges, 
which is where the improvement comes from in the dense regime.

The authors also conjecture that $\log^{3} \left( n \right)$ is the correct order where the transition happens. 
\begin{conjecture}[Bubeck, Ding, Eldan, R\'acz~\cite{BDER16}]\label{conj:sparse}
Let $c > 0$ be fixed and assume $d / \log^{3} \left( n \right) \to \infty$. 
Then 
\[
\TV \left( G\left( n, \frac{c}{n} \right),  G\left( n, \frac{c}{n}, d \right) \right) \to 0.
\]
\end{conjecture}
The main reason for this conjecture is that, 
when $d \gg \log^{3} \left( n \right)$, 
$G(n, c/n)$ and $G(n, c/n, d)$ seem to be locally equivalent; in particular, they both have the same Poisson number of triangles asymptotically. 
Thus the only way to distinguish between them would be to find an emergent global property which is significantly different under the two models, 
but this seems unlikely to exist. 
Proving or disproving this conjecture remains a challenging open problem. 
The best known bound is $n^{3}$ from~\eqref{eq:impossible} (which holds uniformly over $p$).

%

\newpage

\section{Lecture 3: Introduction to entropic central limit theorems and a proof of the fundamental limits of dimension estimation in random geometric graphs} \label{sec:lec3} 

Recall from the previous lecture that the dimension threshold for detecting geometry in $G(n,p,d)$ for constant $p \in (0,1)$ is $d = \Theta \left( n^{3} \right)$. 
What if the random geometric graph model is not $G(n,p,d)$? 
How robust are the results presented in the previous lecture? 
We have seen that the detection threshold is intimately connected to the threshold of when a Wishart matrix becomes GOE. 
Understanding the robustness of this result on random matrices is interesting in its own right, and this is what we will pursue in this lecture.\footnote{This lecture is based on~\cite{BubeckGanguly15}.} 
Doing so also gives us the opportunity to learn about the fascinating world of entropic central limit theorems.

\subsection{Setup and main result: the universality of the threshold dimension}

Let 
$\mathbb{X}$ 
be an $n \times d$ random matrix with i.i.d.\ entries from a distribution $\mu$ that has mean zero and variance $1$. 
The $n \times n$ matrix $\mathbb{X} \mathbb{X}^{T}$ is known as the Wishart matrix with $d$ degrees of freedom. 
As we have seen in the previous lecture, this arises naturally in geometry, where $\mathbb{X} \mathbb{X}^{T}$ is known as the Gram matrix of inner products of $n$ points in $\R^{d}$. 
The Wishart matrix also appears naturally in statistics as the sample covariance matrix, where $d$ is the number of samples and $n$ is the number of parameters.\footnote{In statistics the number of samples is usually denoted by $n$, and the number of parameters is usually denoted by $p$; here our notation is taken with the geometric perspective in mind.} 
We refer to~\cite{BubeckGanguly15} for further applications in quantum physics, wireless communications, and optimization.

We consider the Wishart matrix with the diagonal removed, and scaled appropriately: 
\[
 \mathcal{W}_{n,d} = \frac{1}{\sqrt{d}} \left( \mathbb{X} \mathbb{X}^{T} - \diag \left( \mathbb{X} \mathbb{X}^{T} \right) \right).
\]
In many applications---such as to random graphs, as we have seen in the previous lecture---the diagonal of the matrix is not relevant, so removing it does not lose information. 
Our goal is to understand how large does the dimension $d$ have to be so that $\mathcal{W}_{n,d}$ is approximately like $\mathcal{G}_{n}$, 
which is defined as the $n \times n$ Wigner matrix with zeros on the diagonal and i.i.d.\ standard Gaussians above the diagonal. 
In other words, $\mathcal{G}_{n}$ is drawn from the Gaussian Orthogonal Ensemble (GOE) with the diagonal replaced with zeros.

A simple application of the multivariate central limit theorem gives that if $n$ is fixed and $d \to \infty$, then $\mathcal{W}_{n,d}$ converges to $\mathcal{G}_{n}$ in distribution. 
The main result of Bubeck and Ganguly~\cite{BubeckGanguly15} establishes that this holds as long as $d \, \widetilde{\gg} \, n^{3}$ under rather general conditions on the distribution $\mu$. 
\begin{theorem}[Bubeck and Ganguly~\cite{BubeckGanguly15}]
If the distribution $\mu$ is log-concave\footnote{A measure $\mu$ with density $f$ is said to be log-concave if $f(\cdot) = e^{-\varphi(\cdot)}$ for some convex function $\varphi$.} and $\frac{d}{n^{3} \log^{2} \left( d \right)} \to \infty$, then 
\begin{equation}\label{eq:TV_to_0}
\TV \left( \mathcal{W}_{n,d}, \mathcal{G}_{n} \right) \to 0.
\end{equation}
On the other hand, if $\mu$ has a finite fourth moment and $\frac{d}{n^{3}} \to 0$, then 
\begin{equation}\label{eq:general_signed_triangles}
\TV \left( \mathcal{W}_{n,d}, \mathcal{G}_{n} \right) \to 1.
\end{equation}
\end{theorem}

This result extends Theorems~\ref{thm:gnpd_n3} and~\ref{thm:Wishart_GOE}, and establishes $n^{3}$ as the universal critical dimension (up to logarithmic factors) for sufficiently smooth measures $\mu$: $\mathcal{W}_{n,d}$ is approximately Gaussian if and only if $d$ is much larger than $n^{3}$. 
For random graphs, as seen in Lecture~2, this is the dimension barrier to extracting geometric information from a network: 
if the dimension is much greater than the cube of the number of vertices, then all geometry is lost. 
In the setting of statistics this means that the Gaussian approximation of a Wishart matrix is valid as long as the sample size is much greater than the cube of the number of parameters. 
Note that for some statistics of a Wishart matrix the Gaussian approximation is valid for much smaller sample sizes (e.g., the largest eigenvalue behaves as in the limit even when the number of parameters is on the same order as the sample size~\cite{johnstone2001distribution}).

To distinguish the random matrix ensembles, we have seen in Lecture~2 that signed triangles work up until the threshold dimension in the case when $\mu$ is standard normal. 
It turns out that the same statistic works in this more general setting; 
when the entries of the matrices are centered, this statistic can be written as $A \mapsto \Tr \left( A^{3} \right)$. 
Similarly to the calculations in Section~\ref{sec:signed_triangle}, one can show that under the two measures $\mathcal{W}_{n,d}$ and $\mathcal{G}_{n}$, 
the mean of $\Tr \left( A^{3} \right)$ is $0$ and $\Theta \left( n^{3} / \sqrt{d} \right)$, respectively, 
whereas the variances are $\Theta \left( n^{3} \right)$ and $\Theta \left( n^{3} + n^{5} / d^{2} \right)$, respectively. 
Then~\eqref{eq:general_signed_triangles} follows by an application of Chebyshev's inequality. 
We leave the details as an exercise for the reader.

We note that for~\eqref{eq:TV_to_0} to hold it is necessary to have some smoothness assumption on the distribution $\mu$. 
For instance, if $\mu$ is purely atomic, then so is the distribution of $\mathcal{W}_{n,d}$, and thus its total variation distance to $\mathcal{G}_{n}$ is $1$. 
The log-concave assumption gives this necessary smoothness, and it is an interesting open problem to understand how far this can be relaxed.

\subsection{Pinsker's inequality: from total variation to relative entropy}

Our goal is now to bound the total variation distance 
$\TV \left( \mathcal{W}_{n,d}, \mathcal{G}_{n} \right)$ 
from above. 
In the general setting considered here there is no nice formula for the density of the Wishart ensemble, 
so 
$\TV \left( \mathcal{W}_{n,d}, \mathcal{G}_{n} \right)$ 
cannot be computed directly. 
Coupling these two random matrices also seems challenging. 

In light of these observations, it is natural to switch to a different metric on probability distributions that is easier to handle in this case. 
We refer the reader to the excellent paper~\cite{gibbs2002choosing} which gathers ten different probability metrics and many relations between then. 
Here we use Pinsker's inequality to switch to relative entropy: 
\begin{equation}\label{eq:pinsker}
\TV \left( \mathcal{W}_{n,d}, \mathcal{G}_{n} \right)^{2} 
\leq 
\frac{1}{2} \mathrm{Ent} \left( \mathcal{W}_{n,d} \, \| \, \mathcal{G}_{n} \right),
\end{equation}
where $\mathrm{Ent} \left( \mathcal{W}_{n,d} \, \| \, \mathcal{G}_{n} \right)$ denotes the relative entropy of $\mathcal{W}_{n,d}$ with respect to $\mathcal{G}_{n}$. 
In the following subsection we provide a brief introduction to entropy; the reader familiar with the basics can safely skip this. 
We then turn to entropic central limit theorems and techniques involved in their proof, before finally coming back to bounding the right hand side in~\eqref{eq:pinsker}.

\subsection{A brief introduction to entropy}

The \emph{entropy} of a discrete random variable $X$ taking values in $\mathcal{X}$ is defined as 
\begin{equation*}
 H(X) \equiv H(p) = - \sum_{x \in \mathcal{X}} p \left( x \right) \log \left( p \left( x \right) \right),
\end{equation*}
where $p$ denotes the probability mass function of $X$. 
The $\log$ is commonly taken to have base $2$, in which case entropy is measured in bits; 
if one considers the natural logarithm $\ln$ then it is measured in nats. 
Note that entropy is always nonnegative, since $p(x) \leq 1$ for every $x \in \mathcal{X}$.  
This is a measure of uncertainty of a random variable. It measures how much information is required on average to describe the random variable. 
Many properties of entropy agree with the intuition of what a measure of information should be. 
A useful way of thinking about entropy is the following: if we have an i.i.d.\ sequence of random variables and we know that the source distribution is $p$, then we can construct a code with average description length $H(p)$. 
\begin{example}
If $X$ is uniform on a finite space $\mathcal{X}$, then $H(X) = \log \left| \mathcal{X} \right|$. 
\end{example}

For continuous random variables the \emph{differential entropy} is defined as 
\begin{equation*}
h\left( X \right) \equiv h \left( f \right) = - \int f \left( x \right) \log f \left( x \right) dx,
\end{equation*}
where $f$ is the density of the random variable $X$. 
\begin{example}
If $X$ is uniform on the interval $[0,a]$, then $h(X) = \log \left( a \right)$. 
If $X$ is Gaussian with mean zero and variance $\sigma^{2}$, then $h(X) = \tfrac{1}{2} \log \left( 2 \pi e \sigma^{2} \right)$. 
\end{example}
Note that these examples show that differential entropy can be negative. 
One way to think of differential entropy is to think of $2^{h(X)}$ as ``the volume of the support''.

The \emph{relative entropy} of two distributions $P$ and $Q$ on a discrete space $\mathcal{X}$ is defined as 
\[
 D \left( P \, \| \, Q \right) = \sum_{x \in \mathcal{X}} P(x) \log \frac{P(x)}{Q(x)}.
\]
For two distributions with densities $f$ and $g$ the relative entropy is defined as 
\[
 D \left( f \, \| \, g \right) = \int_{x \in \mathcal{X}} f(x) \log \frac{f(x)}{g(x)}.
\]
Relative entropy is always nonnegative; this follows from Jensen's inequality. 
Relative entropy can be interpreted as a measure of distance between two distributions, although it is not a metric: it is not symmetric and it does not obey the triangle inequality. 
It can be thought of as a measure of inefficiency of assuming that the source distribution is $q$ when it is really $p$. 
If we use a code for distribution $q$ but the source is really from $p$, then we need 
$H(p) + D(p \, \| \, q)$ 
bits on average to describe the random variable.

In the following we use $\mathrm{Ent}$ to denote all notions of entropy and relative entropy. 
We also slightly abuse notation and interchangeably use a random variable or its law in the argument of entropy and relative entropy. 

Entropy and relative entropy satisfy useful chain rules; we leave the proof of the following identities as an exercise for the reader. 
For entropy we have:
\[
 \mathrm{Ent} \left( X_{1}, X_{2} \right) = \mathrm{Ent} \left( X_{1} \right) + \mathrm{Ent} \left( X_{2} \, \middle| \, X_{1} \right).
\]
For relative entropy we have:  
\begin{equation}\label{eq:chain_rel_entropy}
 \mathrm{Ent} \left( \left( Y_{1}, Y_{2} \right) \, \| \, \left( Z_{1}, Z_{2} \right) \right) 
= \mathrm{Ent} \left( Y_{1} \, \| \, Z_{1} \right) 
+ \E_{y \sim \lambda_{1}} \mathrm{Ent} \left( Y_{2} \, | \, Y_{1} = y \, \| \, Z_{2} \, | \, Z_{1} = y  \right), 
\end{equation}
where $\lambda_{1}$ is the marginal distribution of $Y_{1}$ 
and $Y_{2} \, | \, Y_{1} = y$ denotes the distribution of $Y_{2}$ conditionally on the event $\left\{ Y_{1} = y \right\}$.

Let $\phi$ denote the density of $\gamma_{n}$, the $n$-dimensional standard Gaussian distribution, 
and let $f$ be an isotropic density with mean zero, i.e., a density for which the covariance matrix is the identity $I_{n}$. 
Then 
\begin{align*}
0 \leq \mathrm{Ent} \left( f \, \| \, \phi \right) &= \int f \log f - \int f \log \phi \\
&= \int f \log f - \int \phi \log \phi = \mathrm{Ent} \left( \phi \right) - \mathrm{Ent} \left( f \right),
\end{align*}
where the second equality follows from the fact that $\log \phi \left( x \right)$ is quadratic in $x$, and the first two moments of $f$ and $\phi$ are the same by assumption. 
We thus see that the standard Gaussian maximizes entropy among isotropic densities.

\subsection{An introduction to entropic CLTs}

At this point we are ready to state the entropic central limit theorem. 
The central limit theorem states that if $Z_{1}, Z_{2}, \dots$ are i.i.d.\ real-valued random variables with zero mean and unit variance, 
then $S_{m} := \left( Z_{1} + \dots + Z_{m} \right) / \sqrt{m}$ converges in distribution to a standard Gaussian random variable as $m \to \infty$. 
There are many other senses in which $S_{m}$ converges to a standard Gaussian, the entropic CLT being one of them. 
\begin{theorem}[Entropic CLT]
Let $Z_{1}, Z_{2}, \dots$ be i.i.d.\ real-valued random variables with zero mean and unit variance, 
and let $S_{m} := \left( Z_{1} + \dots + Z_{m} \right) / \sqrt{m}$. 
If $\mathrm{Ent} \left( Z_{1} \, \| \, \phi \right) < \infty$, then 
\[
 \mathrm{Ent} \left( S_{m} \right) \nearrow \mathrm{Ent} \left( \phi \right)
\]
as $m \to \infty$. 
Moreover, the entropy of $S_{m}$ increases monotonically, i.e., 
$\mathrm{Ent} \left( S_{m} \right) \leq \mathrm{Ent} \left( S_{m+1} \right)$ 
for every $m \geq 1$. 
\end{theorem}
The condition $\mathrm{Ent} \left( Z_{1} \, \| \, \phi \right) < \infty$ is necessary for an entropic CLT to hold;  
for instance, if the $Z_{i}$ are discrete, then $h \left( S_{m} \right) = - \infty$ for all $m$.

The entropic CLT originates with Shannon in the 1940s and was first proven by Linnik~\cite{Linnik59} in 1959 (without the monotonicity part of the statement). 
The first proofs that gave explicit convergence rates were given independently and at roughly the same time by Artstein, Ball, Barthe, and Naor~\cite{BBN03,ABBN04,ABBN_PTRF04}, and Johnson and Barron~\cite{johnson2004fisher} in the early 2000s, using two different techniques.

The fact that 
$\mathrm{Ent} \left( S_{1} \right) \leq \mathrm{Ent} \left( S_{2} \right)$ 
follows from the entropy power inequality, which goes back to Shannon~\cite{Shannon48} in 1948. 
This implies that 
$\mathrm{Ent} \left( S_{m} \right) \leq \mathrm{Ent} \left( S_{2m} \right)$ 
for all $m \geq 0$, and so it was naturally conjectured that 
$\mathrm{Ent} \left( S_{m} \right)$ 
increases monotonically. 
However, proving this turned out to be challenging. 
Even the inequality 
$\mathrm{Ent} \left( S_{2} \right) \leq \mathrm{Ent} \left( S_{3} \right)$ 
was unknown for over fifty years, 
until Artstein, Ball, Barthe, and Naor~\cite{ABBN04} proved in general that 
$\mathrm{Ent} \left( S_{m} \right) \leq \mathrm{Ent} \left( S_{m+1} \right)$  
for all $m \geq 1$.

In the following we sketch some of the main ideas that go into the proof of these results, in particular following the techniques of Artstein, Ball, Barthe, and Naor~\cite{BBN03,ABBN04,ABBN_PTRF04}.

\subsection{From relative entropy to Fisher information}

Our goal is to show that some random variable $Z$, which is a convolution of many i.i.d.\ random variables, is close to a Gaussian $G$. 
One way to approach this is to \emph{interpolate} between the two. 
There are several ways of doing this; 
for our purposes interpolation along the Ornstein-Uhlenbeck semigroup is most useful. 
Define 
\[
 P_{t} Z := e^{-t} Z + \sqrt{1 - e^{-2t}} G
\]
for $t \in [0,\infty)$, and let $f_{t}$ denote the density of $P_{t} Z$. We have $P_{0} Z = Z$ and $P_{\infty} Z = G$. 
This semigroup has several desirable properties. 
For instance, if the density of $Z$ is isotropic, then so is $f_{t}$. 
Before we can state the next desirable property that we will use, we need to introduce a few more useful quantities.

For a density function $f : \R^{n} \to \R_{+}$, let 
\[
 \mathcal{I} \left( f \right) := \int \frac{\nabla f (\nabla f)^{T}}{f} = \E \left[ \left( \nabla \log f \right) \left( \nabla \log f \right)^{T} \right] 
\]
be the \emph{Fisher information matrix}. 
The Cram\'{e}r-Rao bound states that 
\[
 \Cov \left( f \right) \succeq \mathcal{I} \left( f \right)^{-1}.
\]
More generally this holds for the covariance of any unbiased estimator of the mean. 
The \emph{Fisher information} is defined as 
\[
 I \left( f \right) := \Tr \left( \mathcal{I} \left( f \right) \right). 
\]
It is sometimes more convenient to work with the Fisher information distance, defined as 
$J(f) := I(f) - I(\phi) = I(f) - n$. 
Similarly to the discussion above, one can show that the standard Gaussian minimizes the Fisher information among isotropic densities, and hence the Fisher information distance is always nonnegative.

Now we are ready to state the De Bruijn identity~\cite{Stam59}, which characterizes the change of entropy along the Ornstein-Uhlenbeck semigroup via the Fisher information distance: 
\[
 \partial_{t} \mathrm{Ent} \left( f_{t} \right) = J \left( f_{t} \right).
\]
This implies that the relative entropy between $f$ and $\phi$---which is our quantity of interest---can be expressed as follows: 
\begin{equation}\label{eq:entropy_Fisher}
 \mathrm{Ent} \left( f \, \| \, \phi \right) 
= \mathrm{Ent} \left( \phi \right) - \mathrm{Ent} \left( f \right)
= \int_{0}^{\infty} J \left( f_{t} \right) dt.
\end{equation}
Thus our goal is to bound the Fisher information distance $J(f_{t})$.

\subsection{Bounding the Fisher information distance}

We first recall a classical result by Blachman~\cite{Blachman65} and Stam~\cite{Stam59} that shows that Fisher information decreases under convolution. 
\begin{theorem}[Blachman~\cite{Blachman65}; Stam~\cite{Stam59}]\label{thm:Blachman-Stam}
Let $Y_{1}, \dots, Y_{d}$ be independent random variables taking values in $\R$, 
and let $a \in \R^{d}$ be such that $\left\| a \right\|_{2} = 1$. 
Then 
\[
 I \left( \sum_{i=1}^{d} a_{i} Y_{i} \right) 
\leq \sum_{i=1}^{d} a_{i}^{2} I \left( Y_{i} \right). 
\]
In the i.i.d.\ case, this bound becomes $\left\| a \right\|_{2}^{2} I \left( Y_{1} \right)$. 
\end{theorem}
Artstein, Ball, Barthe, and Naor~\cite{BBN03,ABBN04} gave the following variational characterization of the Fisher information, which gives a particularly simple proof of Theorem~\ref{thm:Blachman-Stam}. 
\begin{theorem}[Variational characterization of Fisher information~\cite{BBN03,ABBN04}]\label{thm:variational_char}
Let $w : \mathbb{R}^{d} \to \left( 0, \infty \right)$ be a sufficiently smooth\footnote{It is enough that $w$ is continuously twice differentiable and satisfies $\int \left\| \nabla w \right\|^{2} / w < \infty$ and $\int \left\| \mathrm{Hess} \left( w \right) \right\| < \infty$.} density on $\mathbb{R}^{d}$, let $a \in \mathbb{R}^{d}$ be a unit vector, and let $h$ be the marginal of $w$ in direction $a$. Then we have 
\begin{equation}\label{eq:variational_char}
I \left( h \right) \leq \int_{\mathbb{R}^{d}} \left( \frac{\mathrm{div} \left( pw \right)}{w} \right)^{2} w
\end{equation}
for any continuously differentiable vector field $p : \mathbb{R}^{d} \to \mathbb{R}^{d}$ with the property that for every $x$, $\left\langle p \left( x \right), a \right\rangle = 1$. 
Moreover, if $w$ satisfies $\int \left\| x \right\|^{2} w \left( x \right) < \infty$, then there is equality for some suitable vector field $p$. 
\end{theorem}
The Blachman-Stam theorem follows from this characterization by taking the constant vector field $p \equiv a$. 
Then we have 
$\mathrm{div} \left( pw \right) = \left\langle \nabla w, a \right\rangle$, 
and so the right hand side of~\eqref{eq:variational_char} becomes 
$a^{T} \mathcal{I} \left( w \right) a$, 
where recall that $\mathcal{I}$ is the Fisher information matrix. 
In the setting of Theorem~\ref{thm:Blachman-Stam} the density $w$ of $\left( Y_{1}, \dots, Y_{d} \right)$ is a product density: 
$w \left( x_{1}, \dots, x_{d} \right) = f_{1} \left( x_{1} \right) \times \dots \times f_{d} \left( x_{d} \right)$, 
where $f_{i}$ is the density of $Y_{i}$. 
Consequently the Fisher information matrix is a diagonal matrix, 
$\mathcal{I} \left( w \right) = \mathrm{diag} \left( I \left( f_{1} \right), \dots, I \left( f_{d} \right) \right)$, 
and thus 
$a^{T} \mathcal{I} \left( w \right) a = \sum_{i=1}^{d} a_{i}^{2} I \left( f_{i} \right)$, 
concluding the proof of Theorem~\ref{thm:Blachman-Stam} using Theorem~\ref{thm:variational_char}.

Given the characterization of Theorem~\ref{thm:variational_char}, one need not take the vector field to be constant; 
one can obtain more by optimizing over the vector field. 
Doing this leads to the following theorem, which gives a rate of decrease of the Fisher information distance under convolutions. 
\begin{theorem}[Artstein, Ball, Barthe, and Naor~\cite{BBN03,ABBN04,ABBN_PTRF04}]\label{thm:ABBN}
Let $Y_{1}, \dots, Y_{d}$ be i.i.d.\ random variables with a density having a positive spectral gap $c$.\footnote{We say that a random variable has spectral gap $c$ if for every sufficiently smooth $g$, we have $\mathrm{Var} \left( g \right) \leq \tfrac{1}{c} \E g'^{2}$. In particular, log-concave random variables have a positive spectral gap, see~\cite{Bobkov99}.} 
Then for any $a \in \mathbb{R}^{d}$ with $ \left\| a \right\|_{2} = 1$ we have that 
\[
 J \left( \sum_{i=1}^{d} a_{i} Y_{i} \right) 
\leq 
\frac{2 \left\| a \right\|_{4}^{4}}{c + (2-c) \left\| a \right\|_{4}^{4}} J \left( Y_{1} \right).
\]
\end{theorem}
When $a = \frac{1}{\sqrt{d}} \mathbf{1}$, then $\frac{2 \left\| a \right\|_{4}^{4}}{c + (2-c) \left\| a \right\|_{4}^{4}} = O \left( 1 / d \right)$, 
and thus using~\eqref{eq:entropy_Fisher} we obtain a rate of convergence of $O \left( 1 / d \right)$ in the entropic CLT. 

A result similar to Theorem~\ref{thm:ABBN} was proven independently and roughly at the same time by Johnson and Barron~\cite{johnson2004fisher} using a different approach involving score functions.

\subsection{A high-dimensional entropic CLT}

The techniques of Artstein, Ball, Barthe, and Naor~\cite{BBN03,ABBN04,ABBN_PTRF04} generalize to higher dimensions, as was recently shown by Bubeck and Ganguly~\cite{BubeckGanguly15}. 
A result similar to Theorem~\ref{thm:ABBN} can be proven, from which a high-dimensional entropic CLT follows, together with a rate of convergence, by using~\eqref{eq:entropy_Fisher} again. 
\begin{theorem}[Bubeck and Ganguly~\cite{BubeckGanguly15}]\label{thm:high_dim_CLT}
Let $Y \in \mathbb{R}^{d}$ be a random vector with i.i.d.\ entries from a distribution $\nu$ with zero mean, unit variance, and spectral gap $c \in (0,1]$. 
Let $A \in \mathbb{R}^{n \times d}$ be a matrix such that $AA^{T} = I_{n}$, the $n \times n$ identity matrix. 
Let $\varepsilon = \max_{i \in [d]} \left( A^T A \right)_{i,i}$ 
and $\zeta = \max_{i,j \in [d], i \neq j} \left| \left( A^T A \right)_{i,j} \right|$. 
Then we have that
\[
\mathrm{Ent} \left( A Y \, \| \, \gamma_{n} \right) 
\leq 
n \min \left\{ 2 \left( \varepsilon + \zeta^{2} d \right) / c, 1 \right\} 
\mathrm{Ent} \left( \nu \, \| \, \gamma_{1} \right), 
\]
where $\gamma_{n}$ denotes the standard Gaussian measure in $\mathrm{R}^{n}$.
\end{theorem}
To interpret this result, consider the case where the matrix $A$ is built by picking rows one after the other uniformly at random on the Euclidean sphere in $\mathbb{R}^{d}$, conditionally on being orthogonal to previous rows (to satisfy the isotropicity condition $AA^{T} = I_{n}$). 
We then expect to have $\varepsilon \simeq n/d$ and $\zeta \simeq \sqrt{n} / d$ (we leave the details as an exercise for the reader), 
and so Theorem~\ref{thm:high_dim_CLT} tells us that 
$\mathrm{Ent} \left( A Y \, \| \, \gamma_{n} \right) \lesssim n^{2}/d$.

\subsection{Back to Wishart and GOE}

We now turn our attention back to bounding the relative entropy 
$\mathrm{Ent} \left( \mathcal{W}_{n,d} \, \| \, \mathcal{G}_{n} \right)$ 
between the $n \times n$ Wishart matrix with $d$ degrees of freedom (with the diagonal removed), $\mathcal{W}_{n,d}$, 
and the $n \times n$ GOE matrix (with the diagonal removed), $\mathcal{G}_{n}$; 
recall~\eqref{eq:pinsker}. 
Since the Wishart matrix contains the (scaled) inner products of $n$ vectors in $\mathbb{R}^{d}$, 
it is natural to relate $\mathcal{W}_{n+1,d}$ and $\mathcal{W}_{n,d}$, since the former comes from the latter by adding an additional $d$-dimensional vector to the $n$ vectors already present. 
Specifically, we have the following:
\[
\mathcal{W}_{n+1,d} = 
\begin{pmatrix}
\mathcal{W}_{n,d} & \frac{1}{\sqrt{d}} \mathbb{X} X \\ 
\frac{1}{\sqrt{d}} \left( \mathbb{X} X \right)^{T} & 0 
\end{pmatrix},
\]
where $X$ is a $d$-dimensional random vector with i.i.d.\ entries from $\mu$, which are also independent from $\mathbb{X}$. 
Similarly we can write the matrix $\mathcal{G}_{n+1}$ using $\mathcal{G}_{n}$: 
\[
\mathcal{G}_{n+1} = 
\begin{pmatrix}
\mathcal{G}_{n} & \gamma_{n} \\ 
\gamma_{n}^{T} & 0 
\end{pmatrix}.
\]

This naturally suggests to use the chain rule for relative entropy and bound 
$\mathrm{Ent} \left( \mathcal{W}_{n,d} \, \| \, \mathcal{G}_{n} \right)$ 
by induction on $n$. 
By~\eqref{eq:chain_rel_entropy} we get that 
\[
\mathrm{Ent} \left( \mathcal{W}_{n+1,d} \, \| \, \mathcal{G}_{n+1} \right) 
= 
\mathrm{Ent} \left( \mathcal{W}_{n,d} \, \| \, \mathcal{G}_{n} \right) 
+ \mathbb{E}_{W_{n,d}} \left[ \mathrm{Ent} \left( \tfrac{1}{\sqrt{d}} \mathbb{X} X \, | \, \mathcal{W}_{n,d} \, \| \, \gamma_{n} \right) \right].
\]
By convexity of the relative entropy we also have that 
\[
\mathbb{E}_{W_{n,d}} \left[ \mathrm{Ent} \left( \tfrac{1}{\sqrt{d}} \mathbb{X} X \, | \, \mathcal{W}_{n,d} \, \| \, \gamma_{n} \right) \right] 
\leq 
\mathbb{E}_{\mathbb{X}} \left[ \mathrm{Ent} \left( \tfrac{1}{\sqrt{d}} \mathbb{X} X \, | \, \mathbb{X} \, \| \, \gamma_{n} \right) \right]. 
\]
Thus our goal is to understand and bound 
$\mathrm{Ent} \left( A X \, \| \, \gamma_{n} \right)$ 
for $A \in \mathbb{R}^{n \times d}$, 
and then apply the bound to $A = \tfrac{1}{\sqrt{d}} \mathbb{X}$ (followed by taking expectation over $\mathbb{X}$). 
This is precisely what was done in Theorem~\ref{thm:high_dim_CLT}, the high-dimensional entropic CLT, for $A$ satisfying $AA^T = I_{n}$. 
Since $A = \tfrac{1}{\sqrt{d}} \mathbb{X}$ does not necessarily satisfy $AA^T = I_{n}$, we have to correct for the lack of isotropicity. 
This is the content of the following lemma, the proof of which we leave as an exercise for the reader. 
\begin{lemma}[\cite{BubeckGanguly15}] 
Let $A \in \mathbb{R}^{n \times d}$ and $Q \in \mathbb{R}^{n \times n}$ be such that 
$QA \left( QA \right)^{T} = I_{n}$. 
Then for any isotropic random variable $X$ taking values in $\mathbb{R}^{d}$ we have that 
\begin{equation}\label{eq:isotropicity}
 \mathrm{Ent} \left( A X \, \| \, \gamma_{n} \right) = \mathrm{Ent} \left( QA X \, \| \, \gamma_{n} \right) + \frac{1}{2} \mathrm{Tr} \left( A A^{T} \right) - \frac{n}{2} + \frac{1}{2} \log \left| \det \left( Q \right) \right|.
\end{equation}
\end{lemma}
We then apply this lemma with 
$A = \tfrac{1}{\sqrt{d}} \mathbb{X}$ 
and 
$Q = \left( \tfrac{1}{d} \mathbb{X} \mathbb{X}^{T} \right)^{-1/2}$. 
Observe that 
$\mathbb{E} \mathrm{Tr} \left( A A^{T} \right) = \tfrac{1}{d} \mathbb{E} \mathrm{Tr} \left( \mathbb{X} \mathbb{X}^{T} \right) = \tfrac{1}{d} \times n \times d = n$, 
and hence in expectation the middle two terms of the right hand side of~\eqref{eq:isotropicity} cancel each other out.

The last term in~\eqref{eq:isotropicity}, 
$- \tfrac{1}{4} \log \det \left( \tfrac{1}{d} \mathbb{X} \mathbb{X}^{T} \right)$, 
should be understood as the relative entropy between a centered Gaussian with covariance given by 
$\tfrac{1}{d} \mathbb{X} \mathbb{X}^{T}$ 
and a standard Gaussian in $\mathbb{R}^{n}$. 
Controlling the expectation of this term requires studying the probability that $\mathbb{X} \mathbb{X}^{T}$ is close to being non-invertible, 
which requires bounds on the left tail of the smallest singular of $\mathbb{X}$. 
Understanding the extreme singular values of random matrices is a fascinating topic, but it is outside of the scope of these notes, and so we refer the reader to~\cite{BubeckGanguly15} for more details on this point.

Finally, the high-dimensional entropic CLT can now be applied to see that 
$\mathrm{Ent} \left( QA X \, \| \, \gamma_{n} \right) \lesssim n^{2} / d$. 
From the induction on $n$ we get another factor of $n$, arriving at 
$\mathrm{Ent} \left( \mathcal{W}_{n,d} \, \| \, \mathcal{G}_{n} \right) \lesssim n^{3} / d$. 
We conclude  that the dimension threshold is $d \approx n^{3}$, 
and the information-theoretic proof that we have outlined sheds light on why this threshold is $n^{3}$.


\newpage

\section{Lectures 4 \& 5: Confidence sets for the root in uniform and preferential attachment trees} \label{sec:lec45} 

In the previous lectures we studied random graph models with community structure and also models with an underlying geometry. 
While these models are important and lead to fascinating problems, they are also static in time. 
Many real-world networks are constantly evolving, and their understanding requires models that reflect this. 
This point of view brings about a host of new interesting and challenging statistical inference questions that concern the temporal dynamics of these networks.

In the last two lectures we will study such questions: 
given the current state of a network, can one infer the state at some previous time? 
Does the initial \emph{seed} graph have an influence on how the network looks at large times? 
If so, is it possible to find the origin of a large growing network? 
We will focus in particular on this latter question. 
More precisely, given a model of a randomly growing graph starting from a single node, called the \emph{root}, we are interested in the following question. 
Given a large graph generated from the model, is it possible to find a small set of vertices for which we can guarantee that the root is in this set with high probability? 
Such \emph{root-finding algorithms} can have applications to finding the origin of an epidemic or a rumor.

\subsection{Models of growing graphs}

A natural general model of randomly growing graphs can be defined as follows. 
For $n \geq k \geq 1$ and a graph $S$ on $k$ vertices, define the random graph $G(n,S)$ by induction. 
First, set $G(k,S) = S$; we call $S$ the \emph{seed} of the graph evolution process. 
Then, given $G(n,S)$, $G(n+1, S)$ is formed from $G(n,S)$ by adding a new vertex and some new edges according to some adaptive rule. 
If $S$ is a single vertex, we write simply $G(n)$ instead of $G(n,S)$. 
There are several rules one can consider; here we study perhaps the two most natural rules: uniform attachment and preferential attachment. 
Moreover, for simplicity we focus on the case of growing \emph{trees}, where at every time step a single edge is added.

Uniform attachment trees are perhaps the simplest model of randomly growing graphs and are defined as follows. 
For $n \geq k \geq 1$ and a tree $S$ on $k$ vertices, the random tree $\UA(n,S)$ is defined as follows. 
First, let $\UA(k,S) = S$. 
Then, given $\UA(n,S)$, $\UA(n+1,S)$ is formed from $\UA(n,S)$ 
by adding a new vertex $u$ and 
adding a new edge $uv$ where the vertex $v$ is chosen uniformly at random among vertices of $\UA \left( n, S \right)$, independently of all past choices. 

\begin{wrapfigure}[13]{r}{0.5\textwidth}
 \centering
 \includegraphics[width=0.35\textwidth]{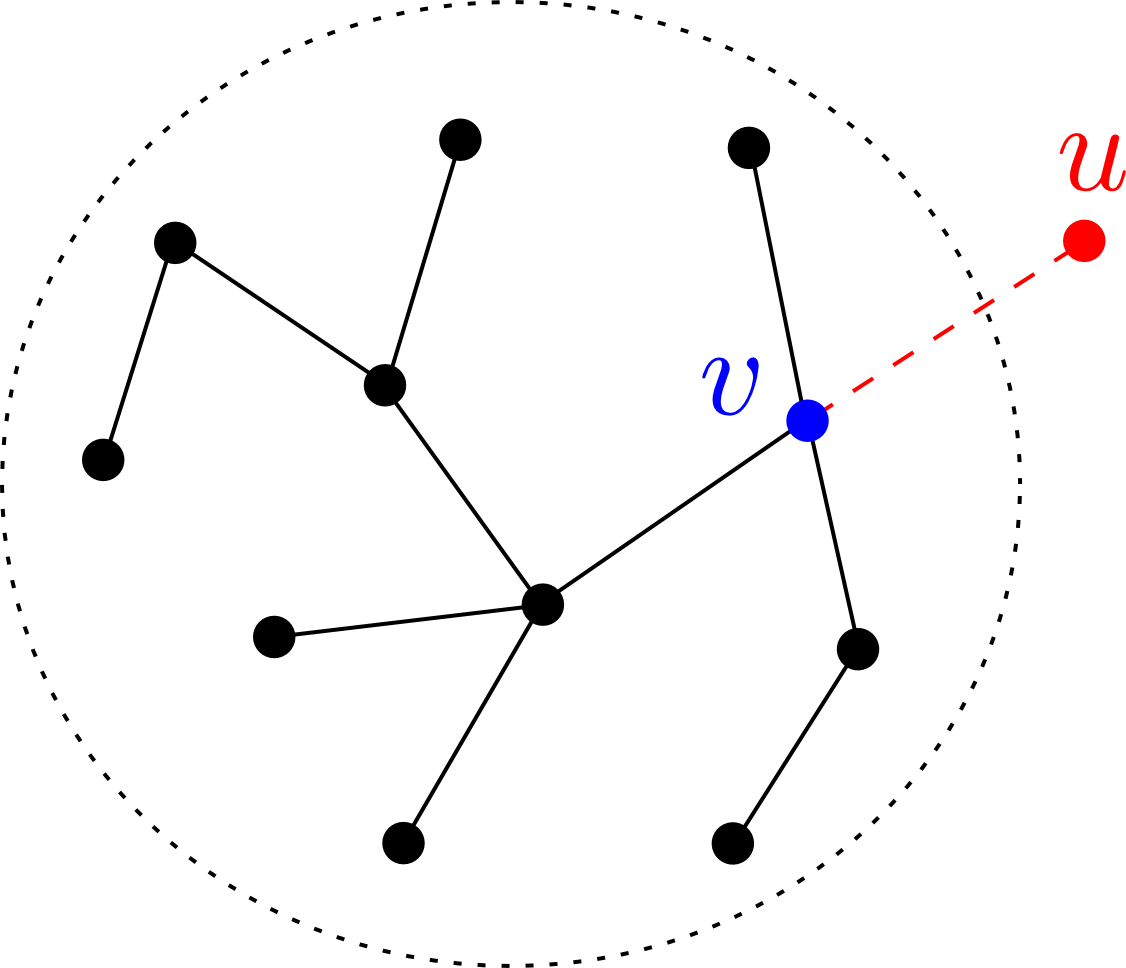}
 \caption{Growing trees: add a new vertex $u$ and attach it to an existing vertex $v$ according some adaptive probabilistic rule.}
 \label{fig:tree_growth}
\end{wrapfigure}
In preferential attachment the vertex is chosen with probability proportional to its degree~\cite{Mah92,BA99,BRST01}. 
For a tree $T$ denote by $d_{T} (u)$ the degree of vertex $u$ in $T$. 
For $n \geq k \geq 2$ and a tree $S$ on $k$ vertices we define the random tree $\PA(n, S)$ by induction. 
First, let $\PA(k, S)=S$. 
Then, given  $\PA(n,S)$, 
$\PA(n+1,S)$ is formed from $\PA(n,S)$ by adding a new vertex $u$ and a new edge $uv$ 
where $v$ is selected at random among vertices in $\PA(n,S)$ according to the following probability distribution: 
\[
\P\left( v = i \, \middle| \, \PA(n, S) \right) 
= \frac{d_{\PA(n, S)}(i)}{2 \left( n - 1 \right)}.
\]

\subsection{Questions: detection and estimation}

The most basic questions to consider are those of 
\emph{detection} and \emph{estimation}. 
Can one detect the influence of the initial seed graph? 
If so, is it possible to estimate the seed? 
Can one find the root if the process was started from a single node? 
We introduce these questions in the general model of randomly growing graphs described above, 
even though we study them in the special cases of uniform and preferential attachment trees later.

The detection question can be rephrased in the terminology of hypothesis testing. 
Given two potential seed graphs $S$ and $T$, and an observation $R$ which is a graph on $n$ vertices, 
one wishes to test whether 
$R \sim G(n, S)$ or 
$R \sim G(n, T)$. 
The question then boils down to whether one can design a test with asymptotically (in $n$) nonnegligible power. 
This is equivalent to studying the total variation distance between $G(n, S)$ and $G(n, T)$, so we naturally define 
\begin{equation*}\label{def:delta}
\delta(S, T) 
:= \lim_{n \to \infty} \mathrm{TV}(G(n, S), G(n, T)),
\end{equation*}
where $G(n,S)$ and $G(n,T)$ are random elements in the finite space of unlabeled graphs with $n$ vertices. 
This limit is well-defined because 
$\mathrm{TV}(G(n, S), G(n, T))$ is nonincreasing in $n$ 
(since if $G(n,S) = G(n,T)$, then the evolution of the random graphs can be coupled such that $G(n', S) = G(n', T)$ for all $n' \geq n$) 
and always nonnegative.

If the seed has an influence, it is natural to ask whether one can estimate $S$ from $G(n,S)$ for large $n$. If so, can the subgraph corresponding to the seed be located in $G(n,S)$? 
We study this latter question in the simple case when the process starts from a single vertex called the \emph{root}.\footnote{In the case of preferential attachment, starting from a single vertex is not well-defined; in this case we start the process from a single edge and the goal is to find one of its endpoints.} 
A \emph{root-finding algorithm} is defined as follows. 
Given $G(n)$ and a target accuracy $\eps \in (0,1)$, 
a root-finding algorithm outputs a set $H\left( G(n), \eps \right)$ of $K(\eps)$ vertices such that 
the root is in $H\left( G(n), \eps \right)$ with probability at least $1-\eps$ (with respect to the random generation of $G(n)$).

An important aspect of this definition is that the size of the output set is allowed to depend on $\eps$, but not on the size $n$ of the input graph. 
Therefore it is not clear that root-finding algorithms exist at all. 
Indeed, there are examples when they do not exist: consider a path that grows by picking one of its two ends at random and extending it by a single edge. 
However, it turns out that in many interesting cases root-finding algorithms do exist. 
In such cases it is natural to ask for the best possible value of $K(\eps)$.

\subsection{The influence of the seed}

Consider distinguishing between 
a preferential attachment tree started from a star with $10$ vertices, $S_{10}$, 
and a preferential attachment tree started from a path with $10$ vertices, $P_{10}$. 
Since the preferential attachment mechanism incorporates the rich-get-richer phenomenon, 
one expects the degree of the center of the star in $\PA(n,S_{10})$ to be significantly larger 
than the degree of any of the initial vertices in the path in $\PA(n,P_{10})$. 
This intuition guided Bubeck, Mossel, and R\'acz~\cite{BMR15} when they initiated the theoretical study of the influence of the seed in preferential attachment trees. 
They showed that this intuition is correct: 
the limiting distribution of the maximum degree of the preferential attachment tree indeed depends on the seed. 
Using this they were able to show that for any two seeds $S$ and $T$ with at least $3$ vertices\footnote{This condition is necessary for a simple reason: the unique tree on $2$ vertices, $S_{2}$, is always followed by the unique tree on $3$ vertices, $S_{3}$, and hence $\delta (S_{2}, S_{3}) = 0$ for any model of randomly growing trees.} and different degree profiles we have 
$\delta_{\PA} (S,T) > 0$.

However, statistics based solely on degrees cannot distinguish all pairs of nonisomorphic seeds. 
This is because if $S$ and $T$ have the same degree profiles, 
then it is possible to couple $\PA(n,S)$ and $\PA(n,T)$ such that they have the same degree profiles for every $n$. 
In order to distinguish between such seeds, it is necessary to incorporate information about the graph structure into the statistics that are studied. 
This was done successfully by Curien, Duquesne, Kortchemski, and Manolescu~\cite{CDKM15}, 
who analyzed statistics that measure the \emph{geometry} of large degree nodes. 
These results can be summarized in the following theorem.

\begin{theorem}\label{thm:main_PA} 
The seed has an influence in preferential attachment trees in the following sense. 
For any trees $S$ and $T$ that are nonisomorphic and have at least $3$ vertices, 
we have $\delta_{\PA}(S,T) > 0$. 
\end{theorem}

In the case of uniform attachment, degrees do not play a special role, so initially one might even think that the seed has no influence in the limit. 
However, it turns out that the right perspective is not to look at degrees but rather the sizes of appropriate subtrees (we shall discuss such statistics later). 
By extending the approach of Curien~et~al.~\cite{CDKM15} to deal with such statistics, Bubeck, Eldan, Mossel, and R\'acz~\cite{BEMR16} showed that the seed has an influence in uniform attachment trees as well.

\begin{theorem}\label{thm:main_UA} 
The seed has an influence in uniform attachment trees in the following sense. 
For any trees $S$ and $T$ that are nonisomorphic and have at least $3$ vertices, 
we have $\delta_{\UA}(S,T) > 0$. 
\end{theorem}

These results, together with a lack of examples showing opposite behavior, suggest that for most models of randomly growing graphs the seed has influence.

\begin{question}\label{q:seed}
How common is the phenomenon observed in Theorems~\ref{thm:main_PA} and~\ref{thm:main_UA}? 
Is there a natural large class of randomly growing graphs for which the seed has an influence?  
That is, models where for any two seeds $S$ and $T$ (perhaps satisfying an extra condition), 
we have $\delta (S,T) > 0$. 
Is there a natural model where the seed has no influence?
\end{question}

The extra condition mentioned in the question could be model-dependent, but should not be too restrictive. 
It would be fascinating to find a natural model where the seed has no influence in a strong sense. 
Even for models where the seed does have an influence, proving the statement in full generality is challenging and interesting.

\subsection{Finding Adam}

These theorems about the influence of the seed open up the problem of \emph{finding} the seed. 
Here we present the results of Bubeck, Devroye, and Lugosi~\cite{BDL16} who first studied root-finding algorithms in the case of uniform attachment and preferential attachment trees. 

They showed that root-finding algorithms indeed exist for preferential attachment trees and that the size of the best confidence set is polynomial in $1/\eps$. 
\begin{theorem}\label{thm:root_PA}
There exists a polynomial time root-finding algorithm for preferential attachment trees with 
$K(\eps) \leq c \tfrac{\log^{2} (1/\eps)}{\eps^{4}}$ 
for some finite constant $c$. 
Furthermore, there exists a positive constant $c'$ such that 
any root-finding algorithm for preferential attachment trees must satisfy 
$K(\eps) \geq \tfrac{c'}{\eps}$. 
\end{theorem}

They also showed the existence of root-finding algorithms for uniform attachment trees. In this model, however, there are confidence sets whose size is \emph{subpolynomial} in $1/\eps$. Moreover, the size of any confidence set has to be at least \emph{superpolylogarithmic} in $1/\eps$. 
\begin{theorem}\label{thm:root_UA}
There exists a polynomial time root-finding algorithm for uniform attachment trees with 
$K(\eps) \leq \exp \left( c \tfrac{\log (1/\eps)}{\log \log (1/\eps)} \right)$ 
for some finite constant $c$. 
Furthermore, there exists a positive constant $c'$ such that 
any root-finding algorithm for uniform attachment trees must satisfy 
$K(\eps) \geq \exp \left( c' \sqrt{\log (1/\eps)} \right)$.
\end{theorem}

These theorems show an interesting quantitative difference between the two models: 
finding the root is exponentially more difficult in preferential attachment than in uniform attachment. 
While this might seem counter-intuitive at first, the reason behind this can be traced back to the rich-get-richer phenomenon: 
the effect of a rare event where not many vertices attach to the root gets amplified by preferential attachment, making it harder to find the root.

In the remaining part of these lectures we explain the basic ideas that go into proving Theorems~\ref{thm:root_PA} and~\ref{thm:root_UA} and prove some simpler special cases. 
Before we do so, we give a primer on P\'olya urns, whose variants appear throughout the proofs. 
If the reader is familiar with P\'olya urns, then the following subsection can be safely skipped.

\subsection{P\'olya urns: the building blocks of growing graph models}\label{sec:polya}

While uniform attachment and preferential attachment are arguably the most basic models of randomly growing graphs, 
the evolution of various simple statistics, such as degrees or subtree sizes, can be described using even simpler building blocks: P\'olya urns. 
This subsection aims to give a brief introduction into the well-studied world of P\'olya urns, while simultaneously showing examples of how these urn models show up in uniform attachment and preferential attachment. 

\subsubsection{The classical P\'olya urn}\label{sec:polya_classical}

The classical P\'olya urn~\cite{Polya23} starts with an urn filled with $b$ blue balls and $r$ red balls. 
Then at every time step you put your hand in the urn, without looking at its contents, and take out a ball sampled uniformly at random. 
You observe the color of the ball, put it back into the urn, together with another ball of the same color. 
This process is illustrated in Figure~\ref{fig:Polya_urn}. 
\begin{figure}[h!]
 \centering
 \includegraphics[width=0.85\textwidth]{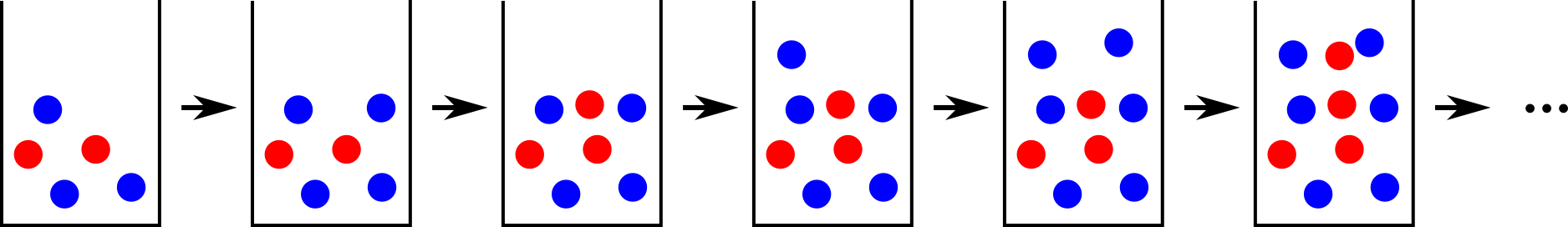}
 \caption{A realization of P\'olya's urn with $b=3$ blue balls and $r=2$ red balls initially.}
 \label{fig:Polya_urn}
\end{figure}

We are interested in the fraction of blue and red balls in the urn at large times. 
Let $X_{n}$ denote the number of blue balls in the urn when there are $n$ balls in the urn in total; 
initially we have $X_{b+r} = b$. 
Furthermore, let $x_{n} = X_{n} / n$ denote the fraction of blue balls when there are $n$ balls in total. 

Let us start by computing the expected increase in the number of blue balls at each time step: 
\[
 \E \left[ X_{n+1} \, \middle| \, X_{n} \right] 
= \left( X_{n} + 1 \right) \times \frac{X_{n}}{n} + X_{n} \times \left( 1 - \frac{X_{n}}{n} \right) 
= \left( 1 + \frac{1}{n} \right) X_{n}.
\]
Dividing this by $\left( n + 1 \right)$ we obtain that 
\[
 \E \left[ x_{n+1} \, \middle| \, \mathcal{F}_{n} \right] = x_{n},
\]
where $\mathcal{F}_{n}$ denotes the filtration of the process up until time $n$ (when there are $n$ balls in the urn); 
since $X_{n}$ is a Markov process, this is equivalent to conditioning on $X_{n}$. 
Thus the fraction of blue balls does not change in expectation; 
in other words, $x_{n}$ is a martingale. 
Since $x_{n}$ is also bounded ($x_{n} \in \left[ 0, 1 \right]$), it follows that $x_{n}$ converges almost surely to a limiting random variable. 
Readers not familiar with martingales should not be discouraged, as it is simple to see heuristically that $x_{n}$ converges: 
when there are $n$ balls in the urn, the change in $x_{n}$ is on the order of $1/n$, 
which converges to zero fast enough that one expects $x_{n}$ to converge.\footnote{The reader can convince themselves that 
$\E\left[ \left( x_{n+1} - x_{n} \right)^{2} \, \middle| \, \mathcal{F}_{n} \right] = \frac{x_{n} \left( 1 - x_{n} \right)}{\left( n + 1 \right)^{2}}$, 
and so the sum of the variances from time $N$ onwards is bounded by $\sum_{n \geq N} \left( n + 1 \right)^{-2} \leq 1/N$.}

Our next goal is to understand the limiting distribution of $x_{n}$. 
First, let us compute the probability of observing the first five draws as in Figure~\ref{fig:Polya_urn}, 
starting with a blue ball, then a red, then two blue ones, and lastly another red: 
this probability is 
$\tfrac{3}{5} \times \tfrac{2}{6} \times \tfrac{4}{7} \times \tfrac{5}{8} \times \tfrac{3}{9}$. 
Notice that the probability of obtaining $3$ blue balls and $2$ red ones in the first $5$ draws is the same regardless of the order in which we draw the balls. 
This property of the sequence $X_{n}$ is known as \emph{exchangeability} and has several useful consequences (most of which we will not explore here). 
It follows that the probability of seeing $k$ blue balls in the first $n$ draws takes on the following form: 
\[
 \p \left( X_{n+b+r} = b + k \right) 
= \binom{n}{k} \frac{b \left( b + 1 \right) \dots \left( b + k - 1 \right) \times r \left( r + 1 \right) \dots \left( r + n - k - 1 \right)}{\left( b + r \right) \left( b + r + 1 \right) \dots \left( b + r + n - 1 \right)}
\]
From this formula one can read off that $X_{n+b+r} - b$ is distributed according to the 
\emph{beta-binomial distribution} with parameters $\left( n, b, r \right)$. 
An alternative way of sampling from the beta-binomial distribution is to first sample a probability $p$ from the beta distribution with parameters $b$ and $r$ 
(having density $x \mapsto \tfrac{\Gamma \left( b + r \right)}{\Gamma \left( b \right) \Gamma \left( r \right)} x^{b-1} \left( 1 -x \right)^{r-1} \mathbf{1}_{\left\{ x \in \left[ 0, 1 \right] \right\}}$), 
and then conditionally on $p$, sample from the binomial distribution with $n$ trials and success probability $p$. 
Conditionally on $p$, the strong law of large numbers applied to the binomial distribution thus tells us that 
$\left( X_{n+b+r} - b \right) / n$ 
converges almost surely to $p$. 
Since $x_{n} = \left( X_{n+b+r} - b \right) / n + o(1)$, it follows that $x_{n} \to p$ almost surely. 
We have thus derived the following theorem. 
\begin{theorem}\label{thm:polya}
Let $x_{n}$ denote the fraction of blue balls at time $n$ (when there are $n$ balls in total) in a classical P\'olya urn which starts with $b$ blue balls and $r$ red balls. 
Then 
\[
 \lim_{n \to \infty} x_{n} = x
\]
almost surely, where $x \sim \Beta \left( b, r \right)$. 
\end{theorem}

\begin{example}
The classical P\'olya urn shows up in uniform attachment trees as it describes the evolution of subtree sizes as follows. 
\begin{figure}[h!]
 \centering
 \includegraphics[width=0.55\textwidth]{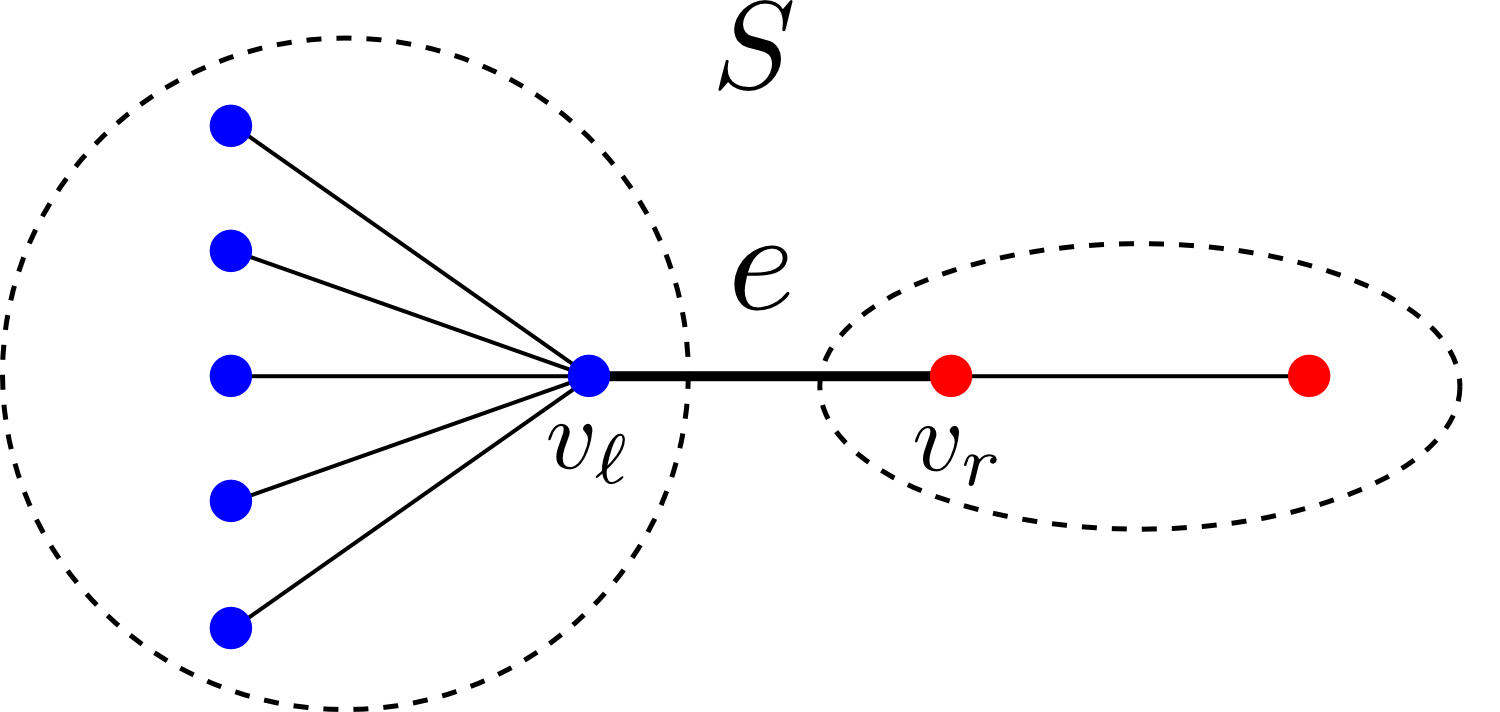}
 \caption{The subtree sizes in uniform attachment trees evolve according to the classical P\'olya urn.}
 \label{fig:Polya_UA}
\end{figure}
Pick an edge of a tree, such as edge $e$ in tree $S$ in Figure~\ref{fig:Polya_UA}, with endpoints $v_{\ell}$ and $v_{r}$. 
This edge partitions the tree into two parts on either side of the edge: a subtree under $v_{\ell}$ and a subtree under $v_{r}$. 
The sizes of these subtrees (i.e., the number of vertices they contain) evolve exactly like the classical P\'olya urn described above (in the example depicted in Figure~\ref{fig:Polya_UA} we have $b=6$ and $r=2$ initially). 
\end{example}

\subsubsection{Multiple colors}\label{sec:Polya_colors}

A natural generalization is to consider multiple colors instead of just two. 
Let $m$ be the number of colors, 
let $\underline{X}_{n} = \left( X_{n,1}, \dots, X_{n,m} \right)$ 
denote the number of balls of each color when there are $n$ balls in the urn in total, 
and let $\underline{x}_{n} = \underline{X}_{n} / n$. 
Assume that initially there are $r_{i}$ balls of color $i$. 

In this case the fraction of balls of each color converges to the natural multivariate generalization of the beta distribution: the Dirichlet distribution. 
The Dirichlet distribution with parameters $\left( r_{1}, \dots, r_{m} \right)$, 
denoted $\Dir \left( r_{1}, \dots, r_{m} \right)$, 
has density 
\[
\underline{x} = \left( x_{1}, \dots, x_{m} \right) \mapsto \tfrac{\Gamma \left( \sum_{i=1}^{m} r_{i} \right)}{\prod_{i=1}^{m} \Gamma \left( r_{i} \right)} x_{1}^{r_{1} - 1} \dots x_{m}^{r_{m}-1} \mathbf{1}_{\left\{ \forall i \, : \, x_{i} \in \left[ 0, 1 \right], \sum_{i=1}^{m} x_{i} = 1 \right\}}.
\]
It has several natural properties that one might expect, for instance the aggregation property, 
that if one groups coordinates $i$ and $j$ together, then the resulting distribution is still Dirichlet, with parameters $r_{i}$ and $r_{j}$ replaced by $r_{i} + r_{j}$. 
This also implies that the univariate marginals are beta distributions. 

The convergence result for multiple colors follows similarly to the one for two colors, so we simply state the result. 
\begin{theorem}\label{thm:polya_dirichlet}
Let $\underline{x}_{n}$ denote the fraction of balls of each color at time $n$ (when there are $n$ balls in total) in a classical P\'olya urn of $m$ colors which starts with $r_{i}$ balls of color $i$.  
Then 
\[
 \lim_{n \to \infty} \underline{x}_{n} = \underline{x}
\]
almost surely, where $\underline{x} \sim \Dir \left( r_{1}, \dots, r_{m} \right)$. 
\end{theorem}

\begin{example}
A P\'olya urn with multiple colors shows up in uniform attachment trees when we partition the tree into multiple subtrees. 
\begin{figure}[h!]
 \centering
 \includegraphics[width=0.55\textwidth]{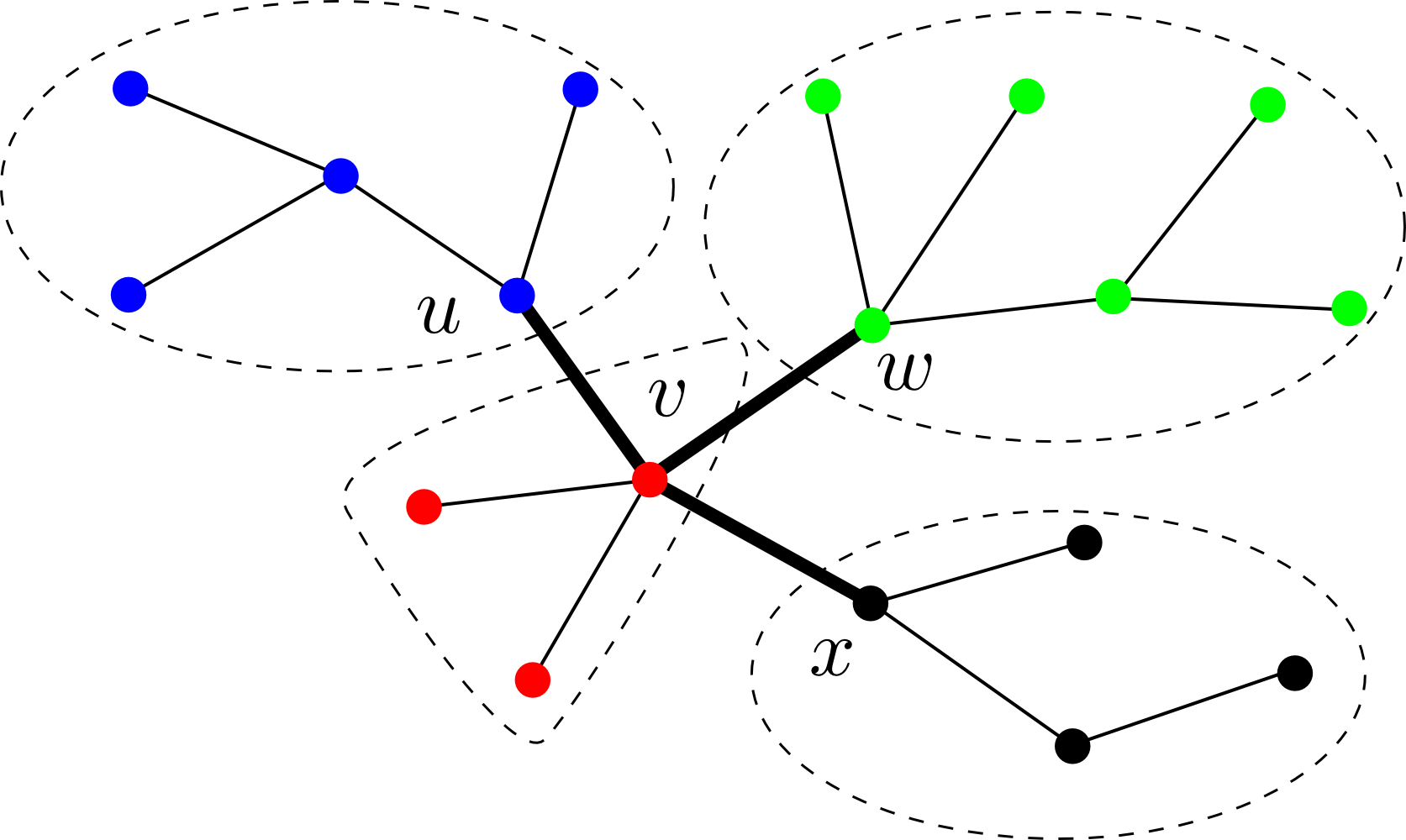}
 \caption{The sizes of multiple subtrees in uniform attachment trees evolve according to a P\'olya urn with multiple colors.}
 \label{fig:Polya_UA_m}
\end{figure}
Picking a subtree of $m$ vertices as highlighted in bold in Figure~\ref{fig:Polya_UA_m}, the tree is partitioned into $m$ subtrees. 
The sizes of these subtrees (i.e., the number of vertices they contain) evolve exactly like the classical P\'olya urn with $m$ colors described above. 
\end{example}

\subsubsection{Adding multiple balls at a time}

It is also natural to consider adding more than one extra ball at each time step. The effect of this is to change the parameter of the limiting Dirichlet distribution. 
\begin{theorem}\label{thm:polya_dirichlet_multiple}
Let $\underline{x}_{n}$ denote the fraction of balls of each color at time $n$ (when there are $n$ balls in total) in a P\'olya urn of $m$ colors which starts with $r_{i}$ balls of color $i$ and where $k$ balls of the same color are added at each time step.  
Then 
\[
 \lim_{n \to \infty} \underline{x}_{n} = \underline{x}
\]
almost surely, where $\underline{x} \sim \Dir \left( r_{1} / k, \dots, r_{m} / k \right)$. 
\end{theorem}

\begin{example}
P\'olya urns where two balls of the same color are added at each time step appear in preferential attachment trees as follows. 
Consider partitioning the tree into $m$ subtrees as in Figure~\ref{fig:Polya_UA_m}, 
but now define the size of a subtree to be the sum of the degrees of the vertices in it. 
Consider which subtree the new incoming vertex attaches to. 
In the preferential attachment process 
each subtree is picked with probability proportional to its size 
and whichever subtree is picked, the sum of the degrees (i.e., the size) increases by $2$ due to the new edge. 
Thus the subtree sizes evolve exactly according to a P\'olya urn described above with $k = 2$. 
\end{example}

\subsubsection{More general urn models}

More generally, one can add some number of balls of each color at each time step. 
The replacement rule is often described by a \emph{replacement matrix} of size $m \times m$, 
where the $i$th row of the matrix describes how many balls of each color to add to the urn if a ball of color $i$ is drawn. 
The urn models studied above correspond to replacement matrices that are a constant multiple of the identity. 
The literature on general replacement matrices is vast and we do not intend to discuss it here; 
our goal is just to describe the simple case when the replacement matrix is 
$\left(
\begin{smallmatrix} 
2 & 0 \\ 
1 & 1 
\end{smallmatrix}
\right)$. 
We refer to~\cite{Jan06} for detailed results on triangular replacement matrices, 
and to the references therein for more general replacement rules.

The urn model with replacement matrix 
$\left(
\begin{smallmatrix} 
2 & 0 \\ 
1 & 1 
\end{smallmatrix}
\right)$ 
can also be described as the classical P\'olya urn with two colors as described in Section~\ref{sec:polya_classical}, 
but in addition a blue ball is always added at each time step. 
It is thus natural to expect that there will be many more blue balls than red balls in the urn at large times. 
It turns out that the number of red balls at time $n$ scales as $\sqrt{n}$ instead of linearly in $n$. 
The following result is a special case of what is proved in~\cite{Jan06}. 
\begin{theorem}\label{thm:Polya_2011}
Let $\left( X_{n}, Y_{n} \right)$ denote the number of blue and red balls, respectively, at time $n$ (when there are $n$ balls in total) 
in a P\'olya urn with replacement matrix 
$\left(
\begin{smallmatrix} 
2 & 0 \\ 
1 & 1 
\end{smallmatrix}
\right)$. 
Assume that initially there are some red balls in the urn. 
Then $Y_{n} / \sqrt{n}$ converges in distribution to a nondegenerate random variable. 
\end{theorem}

\begin{example}
The evolution of the degree of any given vertex in a preferential attachment tree can be understood through such a P\'olya urn. 
More precisely, fix a vertex $v$ in the tree, let $Y_{n}$ denote the degree of $v$ when there are $n$ vertices in total, 
and let $X_{n}$ denote the sum of the degrees of all other vertices. 
Then $\left( X_{n}, Y_{n} \right)$ evolves exactly according to a P\'olya urn with replacement matrix 
$\left(
\begin{smallmatrix} 
2 & 0 \\ 
1 & 1 
\end{smallmatrix}
\right)$. 
This implies that the degree of any fixed vertex scales as $\sqrt{n}$ in the preferential attachment tree. 
\end{example}

\subsection{Proofs using P\'olya urns}

With the background on P\'olya urns covered, we are now ready to understand some of the proofs of the results concerning root-finding algorithms from~\cite{BDL16}.

\subsubsection{A root-finding algorithm based on the centroid}

We start by presenting a simple root-finding algorithm for uniform attachment trees. 
This algorithm is not optimal, but its analysis is simple and highlights the basic ideas.

For a tree $T$, if we remove a vertex $v \in V(T)$, then the tree becomes a forest consisting of disjoint subtrees of the original tree. 
Let $\psi_{T} \left( v \right)$ denote the size (i.e., the number of vertices) of the largest component of this forest. 
For example, in Figure~\ref{fig:Polya_UA} if we remove $v_{r}$ from $S$, then the tree breaks into a singleton and a star consisting of $6$ vertices; thus $\psi_{S} \left( v_{r} \right) = 6$. 
A vertex $v$ that minimizes $\psi_{T} \left( v \right)$ is known as a \emph{centroid} of $T$; one can show that there can be at most two centroids. 
We define the confidence set $H_{\psi}$ by taking the set of $K$ vertices with smallest $\psi$ values. 

\begin{theorem}\label{thm:UA_centroid}\cite{BDL16}
The centroid-based $H_{\psi}$ defined above is a root-finding algorithm for the uniform attachment tree. 
More precisely, 
if $K \geq \tfrac{5}{2} \tfrac{\log \left( 1 / \eps \right)}{\eps}$, 
then 
\[
\liminf_{n \to \infty} \p \left( 1 \in H_{\psi} \left( \UA \left( n \right)^{\circ} \right) \right) 
\geq 
1 - \frac{4\eps}{1-\eps},
\]
where $1$ denotes the root, and $\UA \left( n \right)^{\circ}$ denotes the unlabeled version of $\UA \left( n \right)$. 
\end{theorem}

\begin{proof}
We label the vertices of the uniform attachment tree in chronological order. 
We start by introducing some notation that is useful throughout the proof. 
For $0 \leq i \leq k$, denote by $T_{i,k}$ the tree containing vertex $i$ in the forest obtained by removing in $\UA\left( n \right)$ all edges between vertices $\left\{ 1, \dots, k \right\}$. 
Also, let $\left| T \right|$ denote the size of a tree $T$, i.e., the number of vertices it contains. 
Note that the vector 
$\left( \left| T_{1, k} \right|, \dots, \left| T_{k,k} \right| \right)$ 
evolves according to the classical P\'olya urn with $k$ colors as described in Section~\ref{sec:Polya_colors}, with initial state $\left( 1, \dots, 1 \right)$. 
Therefore, by Theorem~\ref{thm:polya_dirichlet}, the normalized vector 
$\left( \left| T_{1, k} \right|, \dots, \left| T_{k,k} \right| \right) / n$ 
converges in distribution to a Dirichlet distribution with parameters $\left( 1, \dots, 1 \right)$. 

Now observe that 
\[
\p \left( 1 \notin H_{\psi} \right) 
\leq 
\p \left( \exists i > K : \psi \left( i \right) \leq \psi \left( 1 \right) \right) 
\leq 
\p \left( \psi \left( 1 \right) \geq \left( 1 - \eps \right) n \right) 
+ \p \left( \exists i > K : \psi \left( i \right) \leq \left( 1 - \eps \right) n \right). 
\]
We bound the two terms appearing above separately, starting with the first one. 
Note that 
$\psi \left( 1 \right) \leq \max \left\{ \left| T_{1,2} \right|, \left| T_{2,2} \right| \right\}$, 
and both 
$\left| T_{1,2} \right| / n$ 
and 
$\left| T_{2,2} \right| / n$ 
converge in distribution to a uniform random variable in $\left[ 0, 1 \right]$. 
Hence a union bound gives us that 
\[
\limsup_{n \to \infty} \p \left( \psi \left( 1 \right) \geq \left( 1 - \eps \right) n \right) 
\leq 
2 \lim_{n \to \infty} \p \left( \left| T_{1,2} \right| \geq \left( 1 - \eps \right) n \right) 
= 2 \eps. 
\]

For the other term, first observe that for any $i > K$ we have 
\[
 \psi \left( i \right) \geq \min_{1 \leq k \leq K} \sum_{j = 1, j \neq k}^{K} \left| T_{j,K} \right|.
\]
Now using the results on P\'olya urns from Section~\ref{sec:polya} we have that for every $k$ such that $1 \leq k \leq K$, 
the random variable 
$\tfrac{1}{n} \sum_{j = 1, j \neq k}^{K} \left| T_{j,K} \right|$ 
converges in distribution to the $\Beta \left( K - 1, 1 \right)$ distribution. 
Hence by a union bound we have that  
\begin{align*}
\limsup_{n \to \infty} \p \left( \exists i > K : \psi \left( i \right) \leq \left( 1 - \eps \right) n \right) 
&\leq 
\lim_{n \to \infty} \p \left( \exists 1 \leq k \leq K : \sum_{j = 1, j \neq k}^{K} \left| T_{j,K} \right| \leq \left( 1 - \eps \right) n \right) \\
&\leq 
K \left( 1 - \eps \right)^{K - 1}.
\end{align*}
Putting together the two bounds gives that 
\[
 \limsup_{n \to \infty} \p \left( 1 \notin H_{\psi} \right) \leq 2 \eps + K \left( 1 - \eps \right)^{K-1},
\]
which concludes the proof due to the assumption on $K$. 
\end{proof}

The same estimator $H_{\psi}$ works for the preferential attachment tree as well, if one takes 
$K \geq C \frac{\log^{2} \left( 1 / \eps \right)}{\eps^{4}}$ 
for some positive constant $C$. 
The proof mirrors the one above, but involves a few additional steps; we refer to~\cite{BDL16} for details.

For uniform attachment the bound on $K$ given by Theorem~\ref{thm:UA_centroid} is not optimal. 
It turns out that it is possible to write down the maximum likelihood estimator (MLE) for the root in the uniform attachment model; 
we do not do so here, see~\cite{BDL16}. 
One can view the estimator $H_{\psi}$ based on the centroid as a certain ``relaxation'' of the MLE. 
By constructing a certain ``tighter'' relaxation of the MLE, one can obtain a confidence set with size subpolynomial in $1/\eps$ as described in Theorem~\ref{thm:root_UA}. 
The analysis of this is the most technical part of~\cite{BDL16} and we refer to~\cite{BDL16} for more details.

\subsubsection{Lower bounds}

As mentioned above, the MLE for the root can be written down explicitly. 
This aids in showing a lower bound on the size of a confidence set. 
In particular, Bubeck~et~al.~\cite{BDL16} define a set of trees whose probability of occurrence under the uniform attachment model is not too small, 
yet the MLE provably fails, giving the lower bound described in Theorem~\ref{thm:root_UA}.  
We refer to~\cite{BDL16} for details.

On the other hand, for the preferential attachment model it is not necessary to use the structure of the MLE to obtain a lower bound. 
A simple symmetry argument suffices to show the lower bound in Theorem~\ref{thm:root_PA}, which we now sketch.

First observe that the probability of error for the optimal procedure is non-decreasing with $n$, since otherwise one could simulate the process to obtain a better estimate. 
Thus it suffices to show that the optimal procedure must have a probability of error of at least $\eps$ for some finite $n$. 
We show that there is some finite $n$ such that with probability at least $2\eps$, the root is isomorphic to at least $2c / \eps$ vertices in $\PA(n)$. 
Thus if a procedure outputs at most $c/\eps$ vertices, then it must make an error at least half the time (so with probability at least $\eps$). 

Observe that the probability that the root is a leaf in $\PA(n)$ is 
$\tfrac{1}{2} \times \tfrac{3}{4} \times \dots \times \left( 1 - \tfrac{1}{2n} \right) = \Theta \left( 1 / \sqrt{n} \right)$. 
By choosing $n = \Theta \left( 1 / \eps^{2} \right)$, this happens with probability $\Theta \left( \eps \right)$. 
Furthermore, conditioned on the root being a leaf, 
with constant probability vertex $2$ is connected to 
$\Theta \left( \sqrt{n} \right) = \Theta \left( 1 / \eps \right)$ leaves (here we use Theorem~\ref{thm:Polya_2011}), 
which are then isomorphic to the root.

\subsection{Outlook: open problems and extensions}

There are many open problems and further directions that one can pursue;  
the four main papers we have discussed~\cite{BMR15,CDKM15,BEMR16,BDL16} contain $20$ open problems and conjectures alone. 
For instance, can the bounds on the size of the optimal confidence set be improved and ultimately tightened? 
What about other tree growth models? What happens when we lose the tree structure and consider general graphs, e.g., by adding multiple edges at each time step?

When the tree growth model is not as combinatorial as uniform attachment or preferential attachment, 
then other techniques might be useful. 
In particular, many tree growth models can be embedded into continuous time branching processes 
and then the full machinery of general branching processes can be brought to the fore and applied; 
see~\cite{RTV07,Bhamidi07} and the references therein for such results. 
This approach can also be used to obtain finite confidence sets for the root, as demonstrated recently in~\cite{JogLoh16} for sublinear preferential attachment trees.

A closely related problem to those discussed in these lectures is that of 
detecting the source of a diffusion spreading on an underlying network. 
The results are very similar to those above: 
the rumor source can be efficiently detected in many settings, 
see, e.g.,~\cite{shah2011rumors,shah2012finding,KhimLoh15}. 
A different twist on this question is motivated by anonymous messaging services: 
can one \emph{design protocols} for spreading information that preserve anonymity by minimizing the probability of source detection? 
Fanti et al.~\cite{fanti2015spy} introduced a process, termed adaptive diffusion, that indeed achieves this goal. 
Understanding the tradeoffs between privacy and other desiderata is timely and should lead to lots of interesting research.

%


%


\newpage 

\bibliographystyle{abbrv}
\bibliography{bib}




\end{document}